\documentclass[11 pt]{article}  

\usepackage[utf8]{inputenc}
\usepackage{amsmath,indentfirst,qsymbols,graphicx,psfrag,amsfonts,
stmaryrd,color,enumerate,amsthm,thmtools}
\usepackage[normalem]{ulem}
\usepackage[dvipsnames]{xcolor}
\usepackage{hyperref} 

\usepackage{vmargin}
\usepackage{paralist}
\usepackage[many]{tcolorbox}
\setmarginsrb{2.5cm}{2cm}{2.5cm}{2cm}{.4cm}{4mm}{0cm}{7.5mm}

\usepackage[numbers,sort&compress]{natbib}
\usepackage[labelfont={small,bf},font={small,it}]{caption}
\usepackage{amsthm} 

\def \LR{{\sf LR}}  
\def \RR{{\sf RR}}
\def \bpar#1{\left\{\begin{array}{#1} }
\def \epar { \end{array}\right.}

\newtheoremstyle{models}
{}
{}
{}
{}
{\it}
{}
{ }
{\thmname{#1}\thmnumber{ #2}\thmnote{ #3}}
\theoremstyle{models}
\declaretheorem[name=Model,qed={\lower-0.3ex\hbox{$\lhd$}}]{model}
\theoremstyle{theorem}

\def \1{\textbf{1}}

\def \Z{\mathbb{Z}}

\def \app#1#2#3#4#5{\begin{array}{rccl} #1:&#2&\longrightarrow&#3\\ &#4&\longmapsto&#5\end{array}}

\def \N{\mathbb{N}}
\def \R{\mathbb{R}}

\def \cS{{\cal S}}

\def \black{\color{black}}

\def \bar{\overline}

\def \ba{\begin{align}}
\def \ea{\end{align}}
\def \be{\begin{eqnarray*}}
\def \ee{\end{eqnarray*}}
\def \ben{\begin{eqnarray}}
\def \een{\end{eqnarray}}

\def \beq{\begin{equation}}
\def \eq{\end{equation}}

\def \build#1#2#3{\mathrel{\mathop{\kern 0pt#1}\limits_{#2}^{#3}}}

\def \cS{{\cal S}}

\def \dis{\displaystyle}

\def \equi{\Leftrightarrow}
\def \eref#1{(\ref{#1})}

\def \l{\left}

\def \r{\right}
\def \sous#1#2{\mathrel{\mathop{\kern 0pt#1}\limits_{#2}}}
\def \sur#1#2{\mathrel{\mathop{\kern 0pt#1}\limits^{#2}}}
\def \eqd{\sur{=}{(d)}}

\def \E{\mathbb{E}}
\newqsymbol{`B}{\mathcal{B}}
\newqsymbol{`E}{\mathbb{E}}
\newqsymbol{`I}{\mathbb{I}}
\newqsymbol{`N}{\mathbb{N}}
\newqsymbol{`O}{\Omega}
\newqsymbol{`P}{\mathbb{P}}
\newqsymbol{`Q}{\mathbb{Q}}
\newqsymbol{`R}{\mathbb{R}}
\newqsymbol{`W}{\mathbb{W}}
\newqsymbol{`Z}{\mathbb{Z}}
\newqsymbol{`a}{\alpha}
\newqsymbol{`e}{\varepsilon}
\newqsymbol{`o}{\omega}
\newqsymbol{`t}{\tau}
\newqsymbol{`w}{{\cal W}}

\newcommand{\compact}{ \topsep0pt   \itemsep=0pt   \partopsep=0pt   \parsep=0pt}

\newcounter{c}
\def \bir{\begin{itemize}\compact \setcounter{c}{0}}
\def \eir{\end{itemize}\vspace{-2em}~}

\def \bit{\begin{itemize}\compact}
\def \eit{\end{itemize}}

\newcounter{d}
\def \bia{\begin{itemize}\compact \setcounter{d}{0}}
\def \eia{\end{itemize}\vspace{-2em}~}

\newcounter{b}
\def \bi{\begin{itemize}\compact \setcounter{b}{0}}

\def \ei{\end{itemize}\vspace{-2em}~}

\def \bis{\begin{itemize}\compact}

\def \eis{\end{itemize}}

\def \HL{{\sf HL}}

\def \L{{\sf L}}
\def \CT{{\sf CT}}
\def \Y{{\sf Y}}

\begin{document}

\newtheorem{fig}{\hspace{2cm} Figure}
\newtheorem{lem}{Lemma}
\newtheorem{defi}[lem]{Definition}
\newtheorem{pro}[lem]{Proposition}
\newtheorem{theo}[lem]{Theorem}
\newtheorem{cor}[lem]{Corollary}
\newtheorem{note}[lem]{Note}
\newtheorem{conj}{Conjecture}
\newtheorem{Ques}{Question}
\newtheorem*{rem}{Remark}
\renewcommand{\baselinestretch}{1}

\title{\bf The combinatorics of the colliding bullets}
\date{\today}
\author{ Nicolas Broutin\thanks{Sorbonne Université, Sorbonne Paris Cité, CNRS, Laboratoire de Probabilités Statistique et Modélisation, LPSM, F-75005 Paris, France. Email: nicolas.broutin@upmc.fr} \footnote{Both authors are very grateful to the NYU-ECNU Institute of Mathematical Sciences at NYU Shanghai where most of this work has been done.}
        \and Jean-Fran\c{c}ois Marckert\thanks{CNRS, LaBRI, Université Bordeaux} \footnotemark[2]}

\newcommand{\scl}{\text{\sc l}}
\newcommand{\scr}{\text{\sc r}}
\newcommand{\bP}{\mathbf P}
\newcommand{\bD}{\mathbf D}
\newcommand{\cE}{\mathcal E}
\newcommand{\cF}{\mathcal F}
\newcommand{\fS}{\mathfrak S}
\newcommand{\sC}{\mathcal C}
\newcommand{\cR}{\mathcal R}
\newcommand{\cV}{\mathcal V}
\newcommand{\cP}{\mathcal P}

\newcommand{\bV}{\mathbf V}
\newcommand{\bq}{\mathbf q}
\newcommand{\cG}{\mathcal G}
\newcommand{\I}[1]{{\mathbf 1}_{\{#1\}}}
\newcommand{\bDelta}{{\boldsymbol\Delta}}
\newcommand{\nic}{\color{orange}}
\newcommand{\red}{\color{red}}
\newcommand{\alert}[1]{
    \begin{tcolorbox}[
        valign=center,
      enhanced,
      boxrule=0pt,
      arc=0pt,
      outer arc=0pt,
      interior code={\fill[overlay,red] (frame.north west) rectangle (frame.south east);},
    ]\bf #1\end{tcolorbox}
}

\maketitle

\begin{abstract}
The finite colliding bullets problem is the following simple problem: consider a gun, whose barrel remains in a fixed direction; let $(V_i)_{1\le i\le n}$ be an i.i.d.\ family of random variables with uniform distribution on $[0,1]$; shoot $n$ bullets one after another at times $1,2,\dots, n$, where the $i$th bullet has speed $V_i$. When two bullets collide, they both annihilate. We give the distribution of the number of surviving bullets, and in some generalisation of this model. While the distribution is relatively simple (and we found a number of bold claims online), our proof is surprisingly intricate and mixes combinatorial and geometric arguments; we argue that any rigorous argument must very likely be rather elaborate.

\end{abstract}



\section{Introduction}

\subsection{Motivation and models of interest}
\label{sec:motivation}
 
The colliding bullets problem may be stated as follows: a gun whose position and direction remains fixed shoots bullets, one every second. The speeds of the bullets are random, independent and uniform in $[0,1]$. Upon collision, they both annihilate without affecting the others speeds. The main questions of interest concern the distribution of the number of surviving bullets: if we fire $n$ bullets, what is the probability that $k$ bullets escape to infinity? What is the probability if we fire an infinite number of bullets?

\medskip 
The aim of this note is to solve the finite case, as well as some generalizations defined here:



\def \mA{\text{\textsc{ru}}}
\begin{model}\label{def:mA}
{-- \sc Colliding bullets with random speeds and unit delays:}
 $n$ bullets are fired at times $1, 2, \cdots, n $; the respective speeds of the bullet are i.i.d.\ random variables $v_1,\cdots,v_n$ taken under a distribution $\mu$ with support in $[0,+\infty)$ having no atom. Denote by 
\ben
\bP^{\mA}_{n}:=\bP^{\mA}_{n,\mu}=(P_{n}^{\mA}(k),0\leq k \leq n)
\een 
the distribution of the number of surviving bullets. 
\end{model}

\def \mB{\text{\textsc{rr}}}
\begin{model}\label{def:mB}
{-- \sc Colliding bullets with random speeds and random delays:}
$n$ bullets are fired; the respective speeds of the bullets are i.i.d.\ random variables $v_1,\cdots,v_n$ with common distribution $\mu$ with support in $[0,+\infty)$ having no atom. The $i$th bullet is shot at time $T_j=\sum_{k=1}^{j-1} \Delta_k$ where $(\Delta_1,\cdots,\Delta_{n-1})$, the inter-bullet delays, are i.i.d.\ positive random variables (and independent of the speeds) taken under a distribution $\nu$ having no atom at $0$.
Denote by 
\ben
\bP^{\mB}_{n}:=\bP^{\mB}_{n,\mu,\nu}=(P_{n}^{\mB}(k),0\leq k \leq n)
\een 
the distribution of the number of surviving bullets.
\end{model}

\def \mC{\text{\textsc{ff}}}
\begin{model}\label{def:mC}
{-- \sc Colliding bullets with fixed speeds, and fixed delays:}
Choose a vector of $n$ distinct speeds $\bV:=(V_1,\cdots,V_n)\in[0,+\infty)^n$ and $n-1$ inter-bullet delays $\bDelta:=(\Delta_1,\cdots,\Delta_{n-1})$, some positive numbers. Let $\sigma$ and $\tau$ be independent uniform random permutations on $\fS_n$ and $\fS_{n-1}$, respectively. The $i$th bullet has speed $V_{\sigma_i}$ and is shot at time $T^{\tau}_i= \Delta_{\tau_1}+\cdots+\Delta_{\tau_{i-1}}$ (so that $T^{\tau}_1=0$, and the increments of the sequence $(T_i^\tau,1\leq i \leq n)$ are the $\Delta_{\tau_j}$). Assume moreover that for any $(\sigma,\tau)$, $\l[\l(V_{\sigma_i},1\leq i \leq n\r),\l(T^{\tau}_i,1\leq i \leq n\r)\r]$ is generic in the following sense: if we consider that collisions have no effect, there are no pairs $(\sigma,\tau)$, no times, at which three bullets, are exactly at the same place (see formal  Definition \ref{def:gene}). Denote by
\ben
\bP^{\mC}_n=\bP^{\mC}_{\bV,\bDelta}:=\l(P^{\mC}_{\bV,\bDelta}(k),0\leq k \leq n\r)
\een
the distribution of the number of surviving bullets.
\end{model}

 The version of the problem in Model~\ref{def:mC} makes an important link with combinatorics that will be crucial to our approach (see the remark following Theorem~\ref{theo:main} for details).  

\def \mAC{\text{\textsc{faf}}}

\begin{model}\label{def:mD}
{-- \sc Colliding bullets with fixed acceleration functions, and fixed delays:} Fix a continuous increasing function $f:\R^+\to \R^+$ such that $f(0)=0$. Choose an \emph{impetus} vector of $n$ distinct elements ${\bf I}:=(I_1,\cdots,I_n)\in[0,+\infty)^n$ and $n-1$ inter-bullet delays $\bDelta:=(\Delta_1,\cdots,\Delta_{n-1})$, some positive numbers. Let $\sigma$ and $\tau$ be independent uniform random permutations on $\fS_n$ and $\fS_{n-1}$, respectively. For the same convention as in the previous model, the $i$th bullet is shot at time 
$T^{\tau}_i= \Delta_{\tau_1}+\cdots+\Delta_{\tau_{i-1}}.$ The speed of the bullet is not constant (in general): the distance between bullet $i$ and the origin at time $t\geq T^{\tau}_i$ is 
\[D_t(i)= f( I_{\sigma_i} (t- T^{\tau}_i)).\] 
Hence, when it exists, the speed at time $t$ of the $i$th shot bullet is $I_{\sigma_i} f'( I_{\sigma_i} (t- T^{\tau}_i))$ so that for $f(x)=x$ we recover the preceding bullet problem (for $\bV={\bf I}$). For $f(x)=x^2$ we have accelerating bullets with ``constant'' acceleration $2I_{\sigma_i}^2$, for $f(x)=\sqrt{x}$ the asymptotic speed is zero, for $f(x)= 1-\exp(-x)$ the bullets slow down and converge in time $+\infty$ to 1.  Assume again genericity in the sense explained in the previous problem. Denote by
\ben
\bP^{\mAC}_n=\bP^{\mAC,f}_{{\bf I},\bDelta}:=\l(P^{\mAC}_{{\bf I},\bDelta}(k),0\leq k \leq n\r)
\een
the distribution of the number of surviving bullets.
\end{model}

\subsection{Main results and discussion}
\label{sec:results}

For every $n\ge 0$, let $\bq_n$ be the probability distribution that is uniquely characterized by the following recurrence relation and initial conditions:
\ben\label{eq:q01}
q_1(1)=1, \qquad q_1(0)=0, \qquad q_0(0)=1,
\een
and for $N\geq 2$, for any $0\leq k \leq N$,
\ben\label{eq:q02}
q_{N}(k)=\frac{1}{N}q_{N-1}(k-1)+\l(1-\frac1N\r)q_{N-2}(k)
\een
with $q_n(-1)=q_n(k)=0$ if $k>n\geq 0$. 
In other words,  $\bq_N$ is the distribution of $X_N$, where $(X_n,n\geq 0)$ is a simple Markov chain with memory 2 defined by $X_0=0$, $X_1=1$ and for $n\geq 2$, 
\ben\label{eq:law_Xn}
X_n &\eqd& B_{1/n}(1+X_{n-1})+(1-B_{1/n}) X_{n-2}
\een
where $(B_{1/n},n\geq 1)$ is a sequence of independent Bernoulli random variables with respective parameters $1/n$, $n\ge 1$. \medskip 

We are now ready to state our main result:
\begin{theo}\label{theo:main} We have, for any $n\geq 0$,
\[\bP_n^\mA=\bP_n^\mB=\bP_n^\mC=\bP_n^\mAC=\bq_n.\] 
\end{theo}

\begin{rem} (a) Theorem \ref{theo:main} in particular states that the distribution of the number of surviving bullets is independent of $(\bV,\bDelta)$, provided that the colliding bullets problem is well-defined (the probability of a triple collision is zero). One could wonder if this is just the consequence of a much stronger result that would say that the law of the \emph{set of surviving bullets} is independent of $(\bV,\bDelta)$. Exhaustive enumeration all the configurations for examples with few bullets show that the stronger statement is false. \par

(b) As we already mentioned, our analysis will principally rely on the study of ``the permutation model'' $\bP_n^\mC$ for which the pair speeds-delays $(\bV,\bDelta)$ is fixed, but uniformly and independently permuted. When $(\bV,\bDelta)$ is fixed (and generic), the distribution $\bP^{\mC}_{\bV,\bDelta}$ has a combinatorial flavour: the probability that $k$ bullets survive is proportional to the number of permutations $(\sigma,\tau)$ for which this property holds. However, since the distribution of the ``identities'' of the surviving bullets is not the same in general for two different pairs $(\bV',\bDelta')$ and $(\bV,\bDelta)$, the proof of Theorem~\ref{theo:main} cannot rely only on the specifics of the permutations $(\sigma,\tau)$ and must take into account the pair speeds-delays. Our key result is the proof that the map $(\bV,\bDelta)\mapsto \bP^{\mC}_{\bV,\bDelta}$ is constant in the set of generic elements. 

 When one reduces a speed, for example $V_{1}$, the distribution of the set of configurations dramatically changes, and some avalanches of consequences arise (this has somehow the flavour of the \emph{jeu de taquin} used in the Robinston--Schensted correspondence). For some choices of $(\sigma,\tau)$, when one replaces $V_{1}$ by $V_{1}'=V_{1}-`e$, a bullet $A$ which was surviving now collides with another one, say bullet $B$; bullet $B$ which used to collide with bullet $C$, is now destroyed by bullet $D$, etc... The paper is principally devoted to the recursive control of these avalanches of combinatorial modifications that comes from the reduction of a speed.
\end{rem}

The simple form of the recurrence relation for $\bq_n$ also allows us to derive asymptotics for the number of surviving particles as $n$ tends to infinity. The proof relies on a connection between $\bq_n$ and cycles in random permutations that we present in Section~\ref{sec:alt_models}, and we shall present the proof at that point. In the following $\mathcal N(0,1)$ denotes a centered Gaussian random variable with unit variance.

\begin{pro}\label{pro:limit_dist} For $X_n\sim \bq_n$, we have the following convergence in distribution
\[\frac{X_n-\frac 1 2 \log(n)}{\sqrt{\frac 1 2\log(n)}}\xrightarrow[n\to\infty]{(d)} {\cal N}(0,1).\]
\end{pro}

\subsection{Discussion} 
\label{sec:discussion}

Our aim was initially to attack the infinite version, which has been attributed to David Wilson. One quickly realizes that estimates for the probability that all the bullets shot from some time interval vanish (and of related events) should be rather useful when trying to cook up a Borel--Cantelli type argument. This led us to investigate the finite version. In \cite{ibm2014a}, the readers where asked to compute numerically the probability that 20 bullets all annihilate when the speeds are uniform. Apparently motivated by this question, the finite version of the colliding bullets problem arose in a number of online forums (see \cite{stack,stack2}, for instance).  There, the distribution of the number of surviving bullets is \emph{claimed} to be $\bq_n$, and a number of people discuss justifications. However, we were unable to understand the arguments that are claimed to be proofs; this led us develop our own solution (which in the end, seems to be the only rigorous one; see later).

Our proof of Theorem~\ref{theo:main} is rather involved, despite the simplicity of the form of the distributions $\bq_n$. Together with the fact that $\bq_n$ also appears in a number of simpler models that we present in Section~\ref{sec:alt_models}, this may lead the reader to think that there should exist a much shorter and efficient proof. It is possible. Nevertheless, here are some facts that explain why some of the intricate considerations we go through here should be present in any proof:

\begin{itemize}
	\item \emph{Biased permutations in conditional spaces.} First, consider Model \mA. With probability $1/n$ the slowest bullet is fired last and, similarly, also with probability $1/n$ the fastest is fired the first. If one or both of these events occur then the corresponding bullet(s) survive(s). Furthermore, the delays between the remaining bullets are unchanged and all equal to 1. So by exchangeability, the problem reduces to a problem of the same type, and with smaller size. However, if none of these events occur, any decomposition appears to be much more involved: for example, if it is fired after $i-1$ others, for some $i<n$, the slowest bullet may survive or not. When one searches to condition on the survival or destruction of the slowest bullet (or of any other bullet), one quickly realises that the permutation of speeds of the remaining bullets is biased. In the case of eventual survival, the bullets shot after it must collide pairwise, and their space-time diagram must not cross that of the slowest bullet. This creates a bias that looks difficult to handle. 

	\item \emph{Non constant delays.} Even worse, in the case of eventual collision, even if for some reason, (a) the distribution of the speed that collides with it was known, and (b) no other collision occurs first, then after the collision, the delays between the bullets are no more constant; even if we rewind time back to zero, one of the delays is now 3. 

	\item \emph{Random delays?} The previous considerations lead us to study Model \mB, since in any decomposition, the removal of any bullets, will make the delay between the bullets become non constant. However, one also quickly realizes that the distribution of the delay between shots has all the chances to be modified by the removal of two bullets; and the exchangeability might be ruined. 

	\item \emph{Fixed distinct delays must be permuted. }Overall, one is led to consider a sequence of arbitrary delays between the shots. 
	But some computer experiments show that if we replace the unit inter delays $1,\cdots, 1$ by  some deterministic positive and distinct real numbers $\delta_1,\cdots,\delta_{k-1}$, then the distribution of the number of surviving bullets when the bullet speeds are uniformly permuted is not constant and depends on the vector $(\delta_i,1\leq i \leq k-1)$. What appears to be true however (and this is verified by our analysis), is that if one permutes \emph{both} the bullet speeds and the delays, independently, then the distribution of the surviving bullets appears to be the same and given again by $\bq_n$. Hence, Model~\mC~does not appear only as a generalization of Model \mA, but indeed, as a tool to analyze it. 
\end{itemize}

Therefore, when studying the quenched Model~\mC, one of the important points is to guarantee that at every level of the induction, the permutations of both the speeds and of the delays are equally likely and independent, which takes us back to the first point.

When discussing our findings with colleagues, we heard from Vladas Sidoravicius that Fedja Nazarov had an unpublished proof of the fact that, for every $n\ge 1$, one has $q_{2n}(0) = \prod_{i=1}^n (1-\frac 1 {2i})$; this is also stated in \cite{DyJuKiRa2016a}, but the proof is not reproduced there. So, to the best of our knowledge, our proof of Theorem~\ref{theo:main} is the first rigorous treatment of the distribution of the number of surviving bullets in the finite bullets colliding problem.

\subsection{Simpler natural models following the same distribution}
\label{sec:alt_models}

We present here three other models in which the probability distributions $\bq_n$, $n\ge 1$, play a special role. We have found the first two  discussed in a forum about the bullet problem, but we could not find anything that could be considered remotely close to a proof that the colliding bullets problem is indeed equivalent to these models: it seems that some of the users have noticed that the distribution is the same, by a mixture of simulations, exhaustive enumeration, but also what seems to be wrong proofs from the level of details that were provided. In any case, from what we could see, it seems that none of the pitfalls that we have mentioned at the end of the previous section has been dealt with correctly (\cite{stack},\cite{stack2}).

\def \mE{5}
\begin{model}
{-- \sc Sorted bullet flock.} Let $\mu$ be a probability distribution on $(-\infty,+\infty)$ without atoms; $n$ bullets with i.i.d.\ speeds $v_1,\cdots,v_n$ with common distribution $\mu$ are fired at times $1,\cdots, n$. At time 0, the set of living bullet is $L_0=\varnothing$. The $i$th bullet is fired at time $i$ and:
\begin{itemize}
	\item if $v_i\leq \min(L_{i-1})$ then $L_{i}:=L_{i-1}\cup \{v_i\}$; 
	\item if $v_i> \min(L_{i-1})$, $L_i=L_{i-1}\setminus \{\min(L_{i-1})\}$.
\end{itemize}
In other words, if bullet $i$ is faster than one of the surviving bullets, it collides instantaneously with the slowest, and both of them disappear. Otherwise bullet $i$ is the slowest, and it is added to the list of surviving bullets.  
Denote by $\bP_n^{(\mE)}$ the distribution of the number of bullets in the flock at time $n$.
\end{model}

\def \mH{6}
\begin{model}
{-- \sc Odd cycles in random permutations.} Let ${\bf s}$ be a permutation chosen uniformly at random in the symmetric group on $\{1,\cdots,n\}$, and let $Z_n$ be the number of cycles with odd length in the cycle representation of ${\bf s}$. Denote by $\bP_n^{(\mH)}$ the distribution of $Z_n$.
\end{model}


Finally, we introduce the following that ensures a monotone coupling of the marginals at different values of $n$ (it is used in Proposition~\ref{pro:recurrence}):

\def \mAcc{7}
\begin{model}{-- \sc Two-step directed tree on $\mathbb N$.} Let $(B_n, n\ge 1)$ be a sequence of independent Bernoulli random variables with parameters $1/n$. For each $n\ge 1$, if $B_{n}=1$ add the red directed edge $(n,n-1)$, and if $B_n=0$, add the black directed edge $(n,n-2)$. Then, for each node of $\mathbb N$, there is a unique directed edge out, and as a consequence a unique directed path to the node $0$, which is the unique sink of the infinite digraph. For $n\ge 1$, let $\bP^{(\mAcc)}_n$ be the distribution of the number of red edges on the unique path between $n$ and $0$. 
\end{model}

All three models above also follow the exact same distribution $\bq_n$ defined on page~\pageref{eq:q01}: 
\begin{theo}\label{theo:other}For any $\ell\in\{\mE,\mH,  \mAcc\}$, any $n\geq 0$
\[{\bf P}_n^{(\ell)}=\bq_n.\]
\end{theo}

The proofs of the three different cases of Theorem~\ref{theo:other} are all straightforward, and presented in Section~\ref{sec:pf_other}. This contrasts with all bullet-related models of Section~\ref{sec:motivation} for which the proof is fairly intricate.

\subsection{Remarks about the case with infinitely many bullets}

Let $\mu$ be a probability distribution on $[0,\infty)$. Let $(V_i,i\geq 0)$ be a sequence of i.i.d.\ speeds with common distribution $\mu$, and consider the corresponding colliding problem with infinitely many bullets: for each integer $i\ge 0$ a bullet is fired at time $i$, that has speed $V_i$. Let $\cS_\infty$ denote the set of indices of the bullets that survive forever (provided it is well-defined, see later on). Consider now the sequence of colliding problems where only the first $n$ bullets are shot, with speeds $V_1,V_2,\dots, V_{n}$; let $\cS_n$ denote the collection of indices of the surviving bullets.

Theorem~\ref{theo:main} implies that, as $n\to\infty$, the number of surviving bullets $|\cS_n|\to \infty$ in probability, whatever the common distribution $\mu$ of the speeds, provided that it has no atom. 

Of course, without any additional element, this does not rule out the possibility that $|\cS_n|$ may vanish infinitely often. In other words, unsurprisingly, the knowledge of the marginal distributions $(\bq_n)_{n\ge 1}$ alone does not allow to conclude about the eventual survival of some bullet. 

Actually, among the models we have presented in Section~\ref{sec:alt_models}, which all have $\bq_n$ as marginals, some vanish infinitely often, while others tend to infinity almost surely:

\begin{pro}\label{pro:recurrence}
Let $(F_n)_{n\ge 1}$ denote the sequence of sizes in the bullets flock model, and let $(D_n)_{n\ge 1}$ denote the sequence of red distances to $0$ in the two-step tree model. Then, with probability one,
\begin{compactenum}[(i)]
    \item $F_n = 0$ infinitely often, and
    \item $D_n \to \infty$; in particular, $D_n =0$ only finitely often.
\end{compactenum}
\end{pro}

We do not know what happens for the original colliding bullets problem, but at the very least, the sequence $(|\cS_n|)_{n\ge 0}$ exhibits fluctuations that are quite complex. In particular, the fact that $|\cS_n|=0$ does not imply that no bullet from the set $\{1,2,\dots, n\}$ survives forever! Indeed, when bullet $n+1$ is shot, there are three possibilities: it may
\begin{itemize}
  \item be slow enough to avoid all the trajectories of the bullets in $\{1,2,\dots, n\}$, in which case there is one additional surviving bullet,
  \item hit one of the surviving bullets, which would cause the set of surviving bullets to lose one of its elements,
  \item hit some bullet $i\in \{1,\dots, n\}$ that, if not hit, would collide with one of the bullets in $\{1,\dots, i-1\}$, say bullet $j$; the first effect is to release bullet $j$, which may inductively give rise to the same three possibilities.
\end{itemize}
This shows that, for any $n\ge 0$, we have $|\cS_{n+1}|-|\cS_{n}| \in \{+1,-1\}$, but also that  the fluctuations of the sequence $(|\cS_n|)_{n\ge 0}$ are quite complex (see Figure~\ref{fig:dqd} for a simulation), and that the set $\cS_n$ can dramatically change from one step to another. Even showing that if $|\cS_n|=0$ infinitely often then every bullet is eventually destroyed does not seem straightforward. We will leave the question hanging, since we are concerned here with the combinatorics case with only finitely many bullets. 

\begin{figure}[tbp]
\centerline{\includegraphics[scale=0.5]{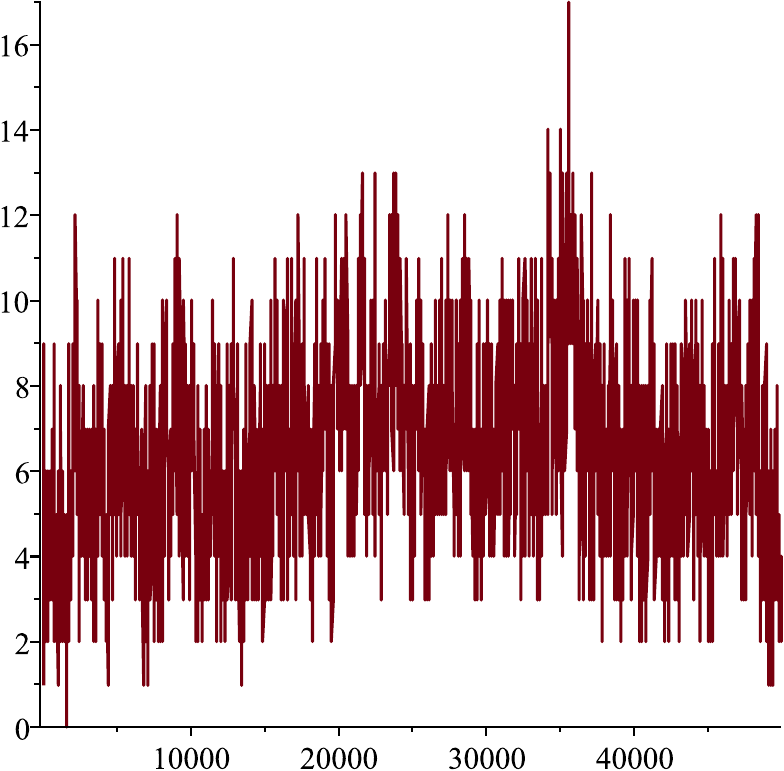}}
\caption{\label{fig:dqd}A simulation of evolution of the number of surviving bullets $(|\cS_j|,1\leq j \leq n)$ for $n=50000$. }
\end{figure} 

Let us just observe that when an infinite number of bullets are shot, the events 
\[\{\textrm{ bullet }i \textrm{ survives }\} 
\qquad \text{and} \qquad
\{~ \exists\, i,  \textrm{ bullet }i \textrm{ survives }\}\]
are measurable. Indeed, a bullet does not survive forever if it is hit before time $m$ for some $m\in \N$; on the other hand, since the speeds are non-negative, the fact that bullet $i\le m$ is hit before time $m$ only depends on the bullets shot before time $m$ (see also Section~\ref{sec:yiluyi}).   

Finally, let us mention that \citet*{DyJuKiRa2016a} have recently proved that in the case where the law of the speeds has support a finite set in $(0,\infty)$, (1) if the first bullet has the second fastest speed then it survives with positive probability (the fastest speed would be obvious), and (2) if it has the slowest speed then it is eventually destroyed with probability one. \citet{SiTo2016a} also have related results. These two papers also contain interesting discussions of the relevant connections with in probability and physics literature. 
To the best of our knowledge, the results in \cite{DyJuKiRa2016a} and \cite{SiTo2016a} are the only non-trivial results on the infinite model that have been proved rigorously.

\begin{rem}Let us mention that the colliding bullet problem has already appeared in the physics literature under the name ``ballistic annihilation''. There compared to the present model, the roles of times and space are exchanged: one sees the bullets as the particles in some gas model; the particles move with constant (random) speeds, initially from different points in space, and the main question is to infer the evolution in the density of particles as times evolves. For more details, see for instance \cite{ElFr1985a,BeRe1993a,KrReLe1995a,BeFe1995a,DrReFrPi1995a,Piasecki1995a,ErToWe1998a,KrSi2001a,Trizac2002b,Trizac2002a}.  
\end{rem}

\subsection{Structure of the proof and plan of the paper} 
\label{sub:sketch_of_the_approach_and_plan_of_the_paper}

A moment thought suffices to see that among Models 1--3 of Section~\ref{sec:motivation}, Model~3 with the fixed parameter $(\bV,\bDelta)$ is the most general; the two others are just annealed versions and the corresponding distributions are obtained by integration. Model~4, with the fixed acceleration function, is actually more general but it may be reduced to Model~3 via coupling and a simple transformation (see Proposition~\ref{pro:auxiliary_models} for details). So from now on, we focus on $\bP_{\bV,\bDelta}^{\mC}$.

Again, the naive induction -- the one relying on the elimination of the slowest or fastest bullet -- fails in general as we have pointed out it Section~\ref{sec:discussion} because one cannot guarantee that a specific bullet collides. There is however a situation in which one can identify a bullet that must either collide or survive: it is when the minimal speed is zero; indeed, in this case, the bullet with minimal speed remains in the barrel so it does survive if it is shot last, and does collide with the next one otherwise. As a consequence, if the minimal speed is null, then there is a simple \emph{one-step reduction} to cases with one or two bullets less. Unfortunately, this trick only works once (at most one speed is zero). But one may try to show that lowering the minimal speed does not alter the distribution of the number of surviving bullets, then one could iteratively bring the minimal speed to zero and thus write down a complete recurrence relation.

We will show that one can indeed alter the parameter $(\bV,\bDelta)$ without changing the distribution of the number of surviving bullets; this will \emph{a posteriori} justify writing $\bP^{\mC}_n$ in place of $\bP_{\bV,\bDelta}^{\mC}$. Granted the independence of $\bP_{\bV,\bDelta}^{\mC}$ from $(\bV,\bDelta)$, it is fairly easy to rigorously devise a recurrence relation for the distribution $\bP_{n}^{\mC}$ and identify it as $\bq_n$, thereby proving Theorem~\ref{theo:main}. We carry on the details in Section~\ref{sub:from_independence_to_the_identification_of_the_law}.

The heart of the argument then consists in justifying the independence of $\bP_{\bV,\bDelta}^{\mC}$ with respect to $(\bV,\bDelta)$. For many values of the parameter $(\bV,\bDelta)$ that we call \emph{generic}, one can indeed slightly lower the minimal speed without affecting the distribution of the number of surviving bullets, for the simple reason that none of the configurations are affected; this is formalized in Lemma~\ref{lem:TCS_open} and relies on a simple topological property. In general however, it is not possible to take the minimal speed to zero without altering any of the configurations. While lowering the minimal speed, one may encounter certain \emph{non-generic} or \emph{singular} values of the parameter $(\bV,\bDelta)$ for which there exist collisions involving more than two bullets. For such a value of the parameter the colliding bullets problem is not well-defined if we consider that bullets only annihilate pairwise. These are values of the parameter at which (at least for some configurations) the pairs of annihilating bullets are modified. Still, in some small neighborhoods of $(\bV,\bDelta)$, the distribution $\bP_{\bV,\bDelta}^{\mC}$ is well-defined and remains constant\footnote{To be precise: the set of generic points restricted to the neighborhood of a singular point is not connected, but $\bP_{\bV,\bDelta}^{\mC}$ is constant on every connected component, and the values all agree.}; because there is no clear one-to-one correspondence between the configurations of these neighborhoods, this invariance necessarily involves a priori intricate averaging. 

The whole argument then reduces to showing that the distribution $\bP^{\mC}_{\bV,\bDelta}$ remains constant when ``crossing'' the non-generic values of the parameter; this is the corner stone of the argument. In order to prove this, we first show that it is ``essentially'' sufficient to consider what happens at the values of the parameter that involve at most triple collisions; these singular values are called \emph{simple}; the formal definition (Definition~\ref{def:simple_singular}) is slightly more restrictive, but this is the idea, and we do not want to blur the big picture at this point. This amounts to showing that it is possible to slightly alter the parameter, without changing the distribution of surviving bullets, in such a way that only simple singular points are encountered when eventually reducing the minimal speed to zero (Lemma~\ref{lem:essentia_generic}). In order to treat the effect of ``crossing'' a simple singular value of the parameter, we proceed by introducing two new combinatorial colliding bullets models with general constraints and by comparing them through an induction argument in Section~\ref{sec:restrictions}.

Turning the sketch we have just presented into a rigorous proof requires to set up a number of geometric representations and formal definitions; these preliminary results are presented in Section~\ref{sec:first_considerations}.


\section{Notation, preliminaries and geometric considerations}
\label{sec:first_considerations} 

\subsection{The virtual space-time diagram}

In this section, we \emph{ignore the effects of collisions} and consider trajectories and events that are only \emph{virtual}. From now on, for $i\ge 1$, we denote by ``bullet $i$'', the $i$th shot bullet, and we assume that the $n$ bullets are shot at some times $t_1<\cdots < t_n$ with respective non-negative speeds $v_1,\cdots,v_n$.  If it is in the air at time $t\in \R$, bullet $i$ lies at a distance to the starting point $x=0$ that is given by
\[\Y_i(t)=v_i(t-t_i).\]
So, ignoring the collisions, each bullet has a \emph{virtual trajectory}; for bullet $i$, it is given by the half-line
\begin{equation}\label{eq:bHL}
\bar \HL\l(v_i,t_i\r):=\{(t,\Y_i(t)), t\geq t_i\}.
\end{equation}
Two bullets $i\ne j$ may only collide at time (which may be positive or negative) 
\ben\label{eq:ff}
T(i,j)=\frac{v_it_i-v_jt_j}{v_i-v_j},
\een
the time at which the lines $\L\l(v_i,t_i\r)$ and $\L\l(v_j,t_j\r)$ supporting respectively $\bar \HL\l(v_i,t_i\r)$ and $\bar \HL\l(v_j,t_j\r)$ intersect.
\begin{figure}[tbp]
\centering 
\includegraphics[width=7.5cm]{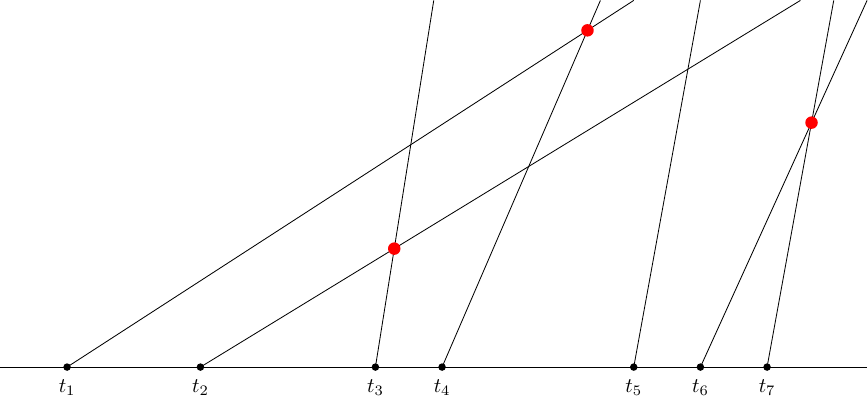}
\hspace{.5cm}
\includegraphics[width=7.5cm]{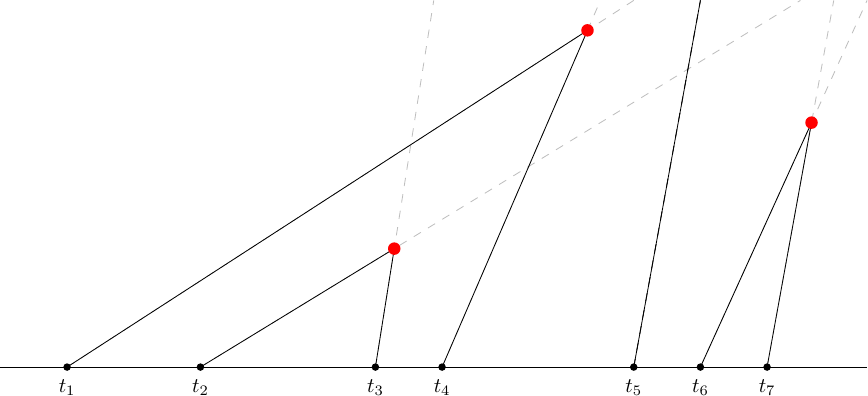}
\caption{\label{fig:space-time}On the left, some instance of a virtual space-time diagram for seven bullets; on the right, the corresponding space-time diagram in solid lines (the virtual space-time diagram is still shown in light dashed lines). The red disks represent the collisions; the collision times are then the first coordinates of the red disks. The 5th bullet survives.}
\end{figure}
We define the \emph{virtual collision time} between bullet $i$ and bullet $j$ by
\ben\label{eq:CT}
\CT(i,j)= \left\{
\begin{array}{ccl} 
  T(i,j) & \textrm{~if~} & T(i,j)\geq \max\{t_i,t_j\}, \\
  +\infty& \textrm{~if~} & T(i,j)< \max\{t_i,t_j\}.
\end{array}
\right.
\een
Even if we ask $T(i,j)$ to be larger than $\max\{t_i,t_j\}$ to ensure that both $i$ and $j$ have been shot, $\CT(i,j)$ is still \emph{virtual}: indeed, \eqref{eq:CT} still ignores the remaining bullets. Even if $T(i,j)<\infty$ there is no guarantee that bullets $i$ and $j$ ever hit each other for one or both may have been destroyed before time $T(i,j)$. The collection of line segments $\{\bar \HL(v_i,t_i)\}_{1\le i\le n}$ is called the \emph{virtual space-time diagram} of the collision process (See Figure~\ref{fig:space-time}).

\subsection{Unambiguous colliding problems and generic parameters}
\label{ssec:genericity}

For $n\ge 1$, let $\Theta_n\subset \R^n_+ \times \R^{n-1}_+$ be the set of couples $(\bV_n,\bDelta_{n-1})$ where $\bV_n=(V_1,\dots, V_n)$ and $\bDelta_{n-1}=(\Delta_1, \dots, \Delta_{n-1})$ for which $0\le V_1 < V_2 < \dots < V_n$ and $0< \Delta_1 \le \Delta_2 \le \dots \le \Delta_{n-1}$. The set $\Theta_n$ will be referred to as the \emph{parameter space} and its elements $(\bV,\bDelta)$ as \emph{parameters}.

A pair $(\sigma,\tau) \in\fS_n\times \fS_{n-1}$ of permutations of the speeds and inter-bullet delays is called a \emph{configuration}. Fix a configuration $(\sigma,\tau)$. Consider
\[\bpar{cccl}
T^{\tau}_j&=&\Delta_{\tau_1}+\cdots+\Delta_{\tau_{j-1}},& ~~\textrm{ for } 1\leq j \leq n,\\
V^{\sigma}_j&=& V_{\sigma_j},                           & ~~\textrm{ for } 1\leq j \leq n.
\epar\]
and, for $t\ge T^\tau_j$,
\ben
\Y_j^{\sigma,\tau}(t)&=& V^\sigma_j(t-T^{\tau}_j)
\een 
the virtual position of bullet $j$ at any time $t\geq T^{\tau}_j$. Let $\CT^{\sigma,\tau}(i,j)$ be the virtual collision time of bullets $i$ and $j$ in configuration $(\sigma,\tau)$; if $\CT^{\sigma,\tau}(i,j)<\infty$ then denote by 
\ben \label{eq:hryj}
M_{\sigma,\tau}^{\bV,\bDelta}(i,j):=\bar\HL\l(V^{\sigma}_i, T_i^\tau\r) \cap \bar \HL\l(V^\sigma_j, T_j^\tau\r)
\een
the corresponding virtual collision point (in space-time) of bullets $i$ and $j$ in the configuration $(\sigma,\tau)$; if $\CT^{\sigma,\tau}(i,j)=\infty$ then the half-lines do not intersect and we set $M_{\sigma,\tau}^{\bV,\bDelta}(i,j)=\varnothing$. Here, all the quantities are still \emph{virtual}, since they are defined independently from the action of the other bullets. 

\begin{defi}[Generic parameter]\label{def:gene}
We say that the parameter $(\bV,\bDelta)\in \Theta_n$ is \emph{generic} if for each fixed configuration $(\sigma,\tau)\in \fS_n\times \fS_{n-1}$, and for any $1\leq i <j <k\leq n$, we have
\ben\label{eq:simple-critic}
\bar \HL\l(V^{\sigma}_i, T_i^\tau\r) \cap \bar \HL\l(V^{\sigma}_j, T_j^\tau\r)\cap \bar \HL\l(V^{\sigma}_k, T_k^\tau\r) =\varnothing.
\een 
Let $\cG_n$ denote the subset of $\Theta_n$ consisting of all generic parameters.
\end{defi}

A parameter is generic if for every configuration no three virtual trajectories ever intersect at the same point. Note that being generic requires more than ``no three bullets meet simultaneously at the same place'', since the constraints hold on \emph{virtual} trajectories. The fact that $(\bV,\bDelta)$ is generic is a sufficient condition for the probability distribution $\bP^{\mC}_{\bV,\bDelta}$ to be defined unambiguously.

\subsection{The set of surviving bullets and the space-time diagram}
\label{sec:yiluyi}

Fix $(\bV, \bDelta)=(\bV_n, \bDelta_{n-1})\in \cG_n$ and a configuration $(\sigma,\tau)$. Then the collection of indices of the bullets that indeed survive is determined by the following simple algorithm. For the sake of readability, we now drop the references to $(\bV,\bDelta)$ and $(\sigma,\tau)$ and suppose that the bullets are shot at times $t_1<t_2<\cdots < t_n$ and have speeds $v_1,v_2,\dots, v_n$.

For each time $t\ge 0$, the algorithm computes the set $\cS_t\subset \{1,2,\dots, n\}$ of bullets that either have not yet been shot, or that are still in the air at time $t$. Initially, we have $\cS_0=\{1,2,\dots, n\}$. Suppose that we have computed $\cS_t$ for $t\le T$. 
If there is a collision between the bullets in $\cS_T$ at some time $t>T$, then a collision occurs at time
\[T^+=\min\{\CT(i,j): i,j\in \cS_T\}.\]
Depending on $T^+$, proceed as follows:
\begin{itemize}
  \item if $T^+=\infty$ there is no collision after time $T$ and therefore $\cS_t=\cS_T$ for all $t\ge T$;
  \item if $T^+<\infty$, then $T^+ = \CT(i_1,j_1)>T$ for some pair $(i_1,j_1)$ of elements of $\cS_T$; this pair is not necessarily unique, but since the parameter is generic, every bullet is involved in at most one colliding pair. Writing $(i_\ell,j_\ell)$, $1\le \ell\le p$ for the pairs for which $\CT(i_\ell,j_\ell)=T^+$, one then has $\cS_{T^+}=\cS_T\setminus \{i_\ell,j_\ell: 1\le \ell\le p\}$ and $\cS_t=\cS_T$ for all $t\in [T,T^+)$.
\end{itemize}
One can thus compute $\cS_t$ for all $t\ge 0$. Since $\cS_t$ is eventually constant, the set of surviving bullets is then $\cS=\lim_{t\to\infty} \cS_t$. The procedure actually computes the time of death $\partial_i$ of each bullet $i\in [n]$. Instead of the completed half-line  $\bar \HL(v_i,t_i)$ one can then consider its \emph{true trajectory} $\HL(v_i,t_i)$ which is the line segment
\begin{equation}\label{eq:HL}
\HL(v_i,t_i)=\{(t,Y_i(t)): t\in [t_i, \partial_i]\}.
\end{equation}
The collection of trajectories $\{\HL(v_i,t_i)\}_{1\le i\le n}$ is called the \emph{space-time diagram}. See the simulations in Fig.~\ref{fig:simu5000}.

\begin{rem}When it is clear from the context which speeds and shooting times we are talking about, we sometimes refer to $\HL(v_i,t_i)$ and $\bar \HL(v_i,t_i)$ as $\HL_i$ and $\bar \HL_i$.
\end{rem}

\begin{figure}[tbp]
\centering
\includegraphics[scale=.45]{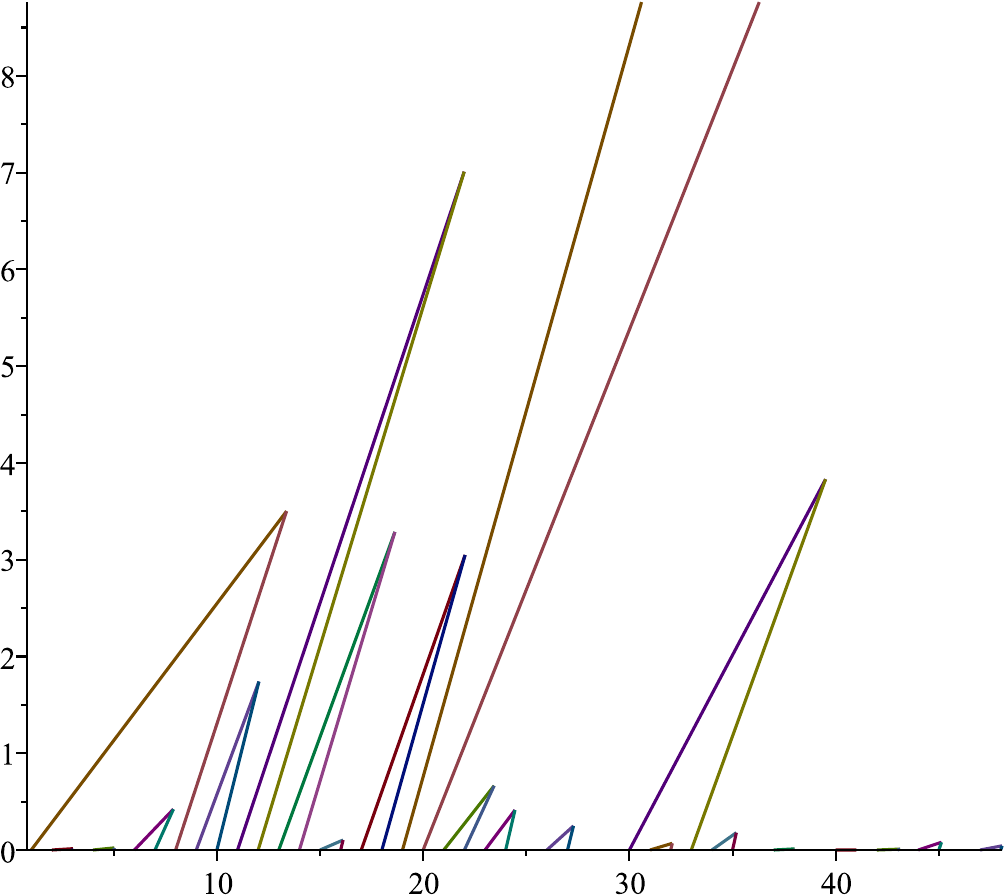}
\includegraphics[scale=.45]{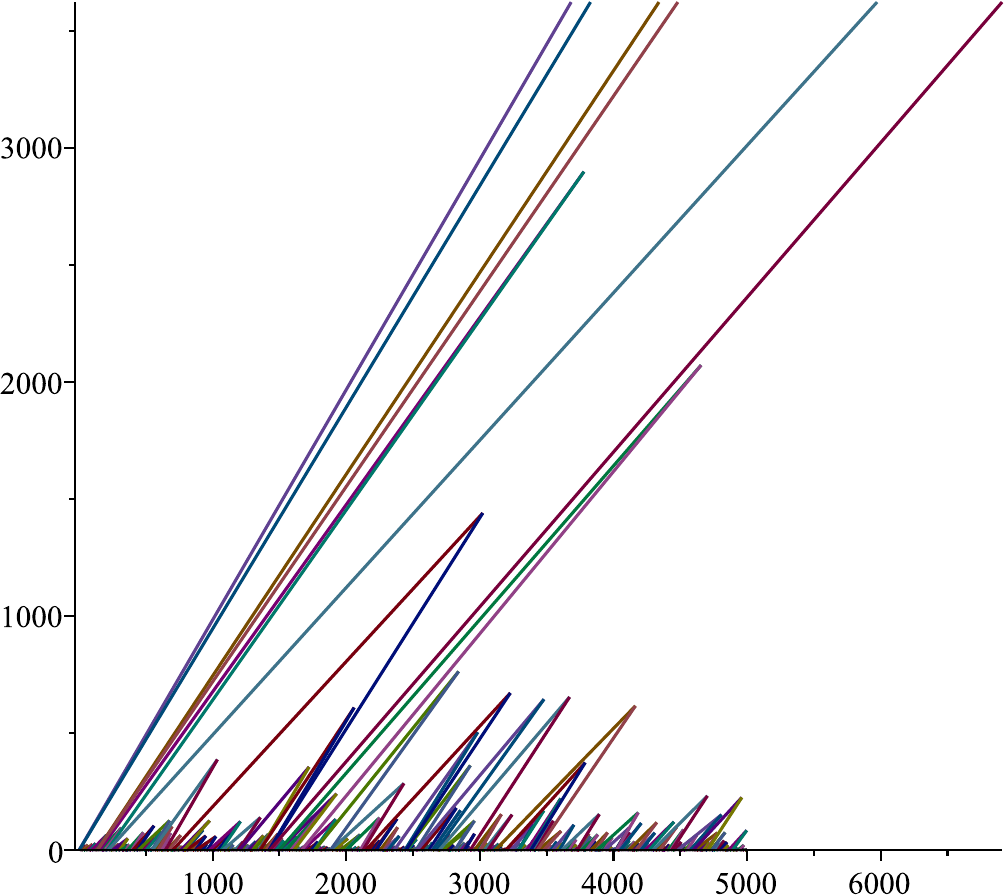}
\caption{\label{fig:simu5000}  Simulations of the space-time diagram in the colliding bullets problem with $n$ bullets, shot with unit delays, with independent uniform random speeds on $[0,1]$. On the left-hand side picture $n=50$, on the right-hand side $n=5000$. }
\end{figure}

\subsection{The topological collision scheme} 
\label{sub:the_topological_collision_scheme}

Fix $(\bV,\bDelta)\in \cG_n$, so that for any configuration $(\sigma,\tau)\in \fS_n\times \fS_{n-1}$, the half-lines $\bar \HL(V^\sigma_i, T_i^\tau)$, $1\le i\le n$, only intersect pairwise. Recall the definition of $M_{\sigma,\tau}^{\bV,\bDelta}(i,j)$ in \eref{eq:hryj} and recall that $\L\l(V^\sigma_k, T_k^\tau\r)$ denotes the line which contains $\bar \HL(V^\sigma_k, T_k^\tau)$.

Each line $\L(V^\sigma_k,T^\tau_k)$ splits $\R^2$ into two open half-spaces; since $V^\sigma_k\in [0,\infty)$, there is a natural labelling each one of these half-spaces as \emph{above} and $\emph{below}$, depending on whether it contains every point $(0,y)$ or $(0,-y)$ for all large enough $y$, respectively. Furthermore, since lines only intersect pairwise, for any triple of distinct integers $(i,j,k)$, the point $M_{\sigma,\tau}^{\bV,\bDelta}(i,j)$, when it exists, lies either above or below the line $\L(V^\sigma_k, T_k^\tau)$.

In the following, we write 
\[\l[n\atop 3\r]^\bullet= 
\left\{
\begin{array}{ll}
\{(i,j,k)\in \{1,\cdots,n\}^3,i<j, j\neq k,i\neq k\} & \text{if } n\ge 3\\
\{(1,2,*)\}& \text{if }n=2\,.
\end{array}
\right.
\]

\begin{defi}[Topological colliding scheme]Let $(\bV,\bDelta)\in \cG_n$. The \emph{topological colliding scheme (TCS, for short) of $(\bV,\bDelta)$} is the function $\Gamma_{\bV,\bDelta}$ defined as follows. 
\begin{enumerate}[(i)]
  \item If $n\ge 3$, then
\ben
\app{\Gamma_{\bV,\bDelta}}
{\fS_n \times \fS_{n-1} \times  \l[n\atop3\r]^\bullet}
{\{-1,0,+1\}^{n!\times (n-1)! \times \binom{n}3}}
{(\sigma
,\tau,i,j,k)}{
\left\{\begin{array}{rl} 
0 & \textrm{if }M_{\sigma,\tau}^{\bV,\bDelta}(i,j) = \varnothing\\
+1 &\textrm{if }M_{\sigma,\tau}^{\bV,\bDelta}(i,j) \textrm{ is above } \L\l(V^\sigma_k, T_k^\tau\r)\\         
-1 & \textrm{if }M_{\sigma,\tau}^{\bV,\bDelta}(i,j) \textrm{ is below } \L\l(V^\sigma_k, T_k^\tau\r)\,,
\end{array}\right.}
\een
\item If $n=2$, then 
$$\Gamma_{\bV,\bDelta}(\sigma,\tau, 1,2,*)=
\left\{
\begin{array}{rl}
0 & \textrm{if }M_{\sigma,\tau}^{\bV,\bDelta}(1,2) = \varnothing\\
1 & \textrm{otherwise.}
\end{array}
\right.$$
\end{enumerate}
\end{defi}

The topological colliding scheme is very similar to order types in geometry. Its importance relies in the (obvious?) fact that for every configuration, it determines indices of the surviving bullets. This is straightforward from the following lemma. 

\begin{lem}\label{lem:tcs}Let $(\bV,\bDelta)\in \cG_n$. For any $(\sigma,\tau)\in \fS_n \times \fS_{n-1}$, the set $\cS(\bV^\sigma,\bDelta^\tau)$ of indices of the surviving bullets in the configuration $(\sigma,\tau)$ is fully determined by the map $\Gamma_{\bV,\bDelta}(\sigma,\tau, \cdot,\cdot,\cdot)$.
\end{lem} 
\begin{proof}We proceed by induction on the number of bullets $n$. If $n\in \{0,1\}$, then there is nothing to prove; if $n=2$, for every permutations $\sigma$ and $\tau$, the point $M^{\bV,\bDelta}_{\sigma,\tau}(i,j)$ exists precisely if $\Gamma_{\bV,\bDelta}(\sigma,\tau, i, j, k)=1$, and hence the information is contained in the restricted map $\Gamma_{\bV,\bDelta}(\sigma,\tau, \cdot, \cdot,\cdot)$ whether the two bullets do collide or not ($\bDelta$ has no influence).

Suppose now that the property holds for up to $n-1$ bullets, and any generic pair $(\bV,\bDelta)$. 
Fix a pair of permutations $(\sigma,\tau)\in \fS_{n}\times \fS_{n-1}$.
Observe first that considering a subset of the trajectories $\HL_i$, $i\in [n]$, does correspond to looking at some TCS for a speed vector consisting of the speeds of the selected bullets, and a delay vector containing aggregate delays; if $(\bV,\bDelta)$ is generic, so are any of the vectors obtained by taking subsets of the bullet trajectories. 
There are two possibilities: 
\begin{itemize}
  \item if $\Gamma_{\bV,\bDelta}(\sigma,\tau, 1, \cdot, \cdot )=0$ then $\HL_1$ does not intersect any of other $\HL_i$. In this case $1\in \cS_{\bV^\sigma,\bDelta^\tau}$ and the remaining surviving bullets are determined by the map induced by $\Gamma_{\bV,\bDelta}(\sigma,\tau, \cdot,\cdot,\cdot)$ on the set $\begin{bmatrix}B\atop3\end{bmatrix}^\bullet$ where $B=\{2,\dots, n\}$. 
  \item otherwise, there exists some $j$ such that $\Gamma_{\bV,\bDelta}(\sigma,\tau,1,j, \cdot)$ is not identically zero. Then, let $J\in \{2,3,\dots, n\}$ be minimal such that $\Gamma_{\bV,\bDelta}(\sigma,\tau, 1,J, k)=1$, for all $k\not\in \{1,J\}$ (so that $M^{\bV,\bDelta}_{\sigma,\tau}(1,J)$ lies above all lines $\L_i$, $i\in \{2,\dots, n\}\setminus \{J\}$). By induction, we can determine whether bullet $J$ survives when removing bullets with indices in the set $\{1, J+1,J+2,\dots, n\}$ by looking at the map $\Gamma_{\bV,\bDelta}(\sigma,\tau, \cdot,\cdot,\cdot)$ on the set $\begin{bmatrix}B\atop3\end{bmatrix}^\bullet$ where now $B=\{2,\dots, J\}$. With this information in hand: 
  \begin{itemize}
    \item If bullet $J$ does survive in this smaller colliding scheme, then bullets $1$ and $J$ do collide in the original scheme. Additionally, bullets $2,3,\dots, J-1$ all annihilate and none of the trajectories genuinely crosses $\HL_J$. Another induction yields the indices of the surviving bullets lying in $\{J+1,\dots, n\}$ by looking at $\Gamma_{\bV,\bDelta}(\sigma,\tau, \cdot,\cdot,\cdot)$ on the set $\begin{bmatrix}B\atop3\end{bmatrix}^\bullet$ with $B=\{J+1,\dots, n\}$.
    \item If bullet $J$ does not survive, it collides with another one (whose index is given by the induction); removing both, we can again use induction to determine the remaining surviving bullets.
  \end{itemize}
\end{itemize}
It follows that the set of indices of the surviving bullets $\cS(\bV^\sigma,\bDelta^\tau)$ is a function of the map $\Gamma_{\bV,\bDelta}(\sigma,\tau, \cdot,\cdot, \cdot)$.
\end{proof}

The following simple observation will also be useful:
\begin{lem}\label{lem:TCS_open}
  The map $\Gamma: (\bV,\bDelta)\mapsto \Gamma_{\bV,\bDelta}$ is locally constant in $\cG_n$: for $(\bV,\bDelta)\in \cG_n$, there exists an open neighborhood ${\cal O}$ of $(\bV,\bDelta)$ in $\R^n\times \R^{n-1}$ such that ${\cal O}\subset {\cal G}_n$,
 and
\ben\label{eq:fdef}
\Gamma_{\bV,\bDelta}=\Gamma_{\bV',\bDelta'}\qquad \text{for all } (\bV',\bDelta')\in \cal O.
\een
\end{lem}

\subsection{Singular parameters and critical patterns}
\label{ssec:singular-critical} 

A parameter $(\bV_n,\bDelta_{n-1})$ in $\Theta_n\setminus \cG_n$ is called \emph{singular}. The parameter $(\bV_n, \bDelta_{n-1})$ is singular if and only it contains a critical pattern, or critical bi-triangle in the following sense: 
\begin{defi}[Critical pattern]A \emph{critical pattern} or \emph{critical bi-triangle} with respect to some parameter $(\bV,\bDelta)$ is a tuple $(v_m, v_\ell,v_r, d_\ell,d_r)$ such that
\begin{compactenum}[(a)]
  \item $v_m$, $v_\ell$ and $v_r$ are three distinct speeds from $\bV$,
  \item $d_\ell$ and $d_r$ are the sums of components of $\bDelta$ over two disjoint sets of indices, and
  \item the three half-lines $\bar \HL(v_{m},0)$, $\bar \HL(v_\ell,d_\ell)$, and $\bar \HL(v_r,d_\ell+d_r)$ are concurrent (see Figure~\ref{fig:FDN}). 
\end{compactenum}
\end{defi}

Given $(\bV,\bDelta)$, a configuration $(\sigma,\tau)$ is said to \emph{contain the critical pattern} $\pi=(v_m,v_l,v_r, d_l,d_r)$ if there exists $(i,j,k)$ such that $V^\sigma_i=v_m$, $V^\sigma_j=v_\ell$, $V^\sigma_k=v_r$ and furthermore
\be
\bpar{ccl}
T^\tau_j-T^\tau_i=\dis\sum_{p=i}^{j-1} \Delta^\tau_p =d_\ell\\
T^\tau_k-T^\tau_j=\dis\sum_{p=j}^{k-1} \Delta^\tau_p =d_r;
\epar
\ee
we let $\cal C_\pi=\cal C_\pi(\bV,\bDelta)$ denote the set of configurations that contain $\pi$. For a given configuration $(\sigma,\tau)\in \cal C_\pi$ containing a given critical pattern $\pi=(v_m, v_\ell,v_r,d_\ell,d_r)$, { let $\rho_\pi\ge 0$ be such that the common intersection point of $\bar \HL(v_{m},0)$, $\bar \HL(v_\ell,d_\ell)$, and $\bar \HL(v_r,d_\ell+d_r)$ has coordinates $(t,\rho_\pi)$, for some $t\ge 0$\footnote{So $\rho_\pi$ is the (spatial) distance from the origin at which the lines intersect; this is independent of the potential time shift that the pattern may have in specific configuration.}.
 The pattern is called \emph{realized} by $(\sigma,\tau)$ if it is also contained in the diagram involving the actual trajectories $\HL(V^\sigma_i, T^\tau_i)$; this means that for $p\in \{i,j,k\}$, the actual trajectory $\HL(V^\sigma_p,T^\tau_p)$ contains the line segment $\bar \HL(V^\sigma_p,T^\tau_p)\cap [0,\infty) \times [0,\rho_\pi)$ (meaning that the portions of the half-lines before the triple collision point are not intersected, see Figure~\ref{fig:FDN}); we let $\cal R_\pi=\cal R_\pi(\bV,\bDelta)$ be the set of configurations for which $\pi$ is realized. Finally, we say that a critical pattern $\pi=(v_m,v_\ell,v_r, d_\ell, d_r)$ is \emph{minimal} if both $d_\ell$ and $d_r$ correspond to a single component of $\bDelta$ (meaning that the three bullets with speed $v_m, v_r$ and $v_\ell$ are shot consecutively).

{While a given speed $v$ may be involved in multiple critical patterns, there may not be a single configuration that contains multiple critical patterns. For this reason, it is useful to keep track of which configurations contain a given critical pattern: a pattern is called $(\sigma,\tau)$-critical if it is critical for $(\sigma,\tau)$.}

\begin{defi}[Simple singular parameter]\label{def:simple_singular}
A singular parameter $(\bV_n,\bDelta_{n-1})$ is called \emph{simple}, if for every configuration $(\sigma,\tau)$, every speed $v$ of $\bV_n$ is involved in at most one $(\sigma,\tau)$-critical pattern. 
In other words, for each $(\sigma,\tau)$, each line  $\HL(V^{\sigma}_m, T_m^\tau)$ participates in at most one critical bi-triangle; in particular, the collisions involve at most three bullets.
\end{defi}

For a speed vector $\bV_n=(V_1,V_2,\dots, V_n) \in \R^n_+$, let $\bV_n^{\downarrow}$ be the vector $(0,V_2,\dots, V_n)$ where the minimal speed has been put to zero.

\begin{defi}[Essentially generic]\label{def:essentially_generic}
A parameter $(\bV_n,\bDelta_{n-1})\in \cG_n$ is called \emph{essentially generic} if, for any convex combination $\bV_n'$ of $\bV_n$ and $\bV^\downarrow_n$, that is $\bV_n'=(\lambda V_1,V_2,\dots, V_n)$ for some some $\lambda \in [0,1]$, the parameter $(\bV'_n,\bDelta_{n-1})$ is either generic or simple singular.
\end{defi}


When $(\bV,\bDelta)$ is essentially generic, every critical pattern with respect to a convex combination $(\bV',\bDelta)$ of $(\bV,\bDelta)$ and $(\bV^\downarrow,\bDelta)$ must involve the minimal speed: it must be of the form $(\min \bV',v_\ell, v_r,d_\ell, d_r)$. {The following crucial ``density lemma'' allows us to focus only on essentially generic parameters:
\begin{lem}\label{lem:essentia_generic}
For any $(\bV,\bDelta)\in \cG_n$, there exists an essentially generic parameter $(\bV',\bDelta')\in \cG_n$ such that the TCS of $(\bV, \bDelta)$ and $(\bV',\bDelta')$ are identical, \emph{i.e.}, $\Gamma_{\bV,\bDelta} = \Gamma_{\bV',\bDelta'}$.
\end{lem}
Our proof of Lemma~\ref{lem:essentia_generic} is probabilistic; for the readers who might be averse to such an existential proof, we mention that one could alternatively give  an explicit and deterministic construction of the parameter $(\bV',\bDelta')$.
\begin{proof}By Lemma~\ref{lem:TCS_open}, there exists a full-dimensional compact set $J\times K \subset \R^n_+\times \R^{n-1}_+$ around $(\bV,\bDelta)$ in which the TCS is constant; in particular, $J\times K$ is included in $\cG_n$. We may, and will from now on, assume that $J$ is a rectangular box. 

  Furthermore, this implies that for any such point $(\bV^\circ,\bDelta^\circ)$, lowering the minimal speed can create at most triple-intersections in the virtual space-time diagram. In order to complete the proof, it suffices to prove that there exists some point in $J$ for which no two such triple collision points are ever aligned with the point in space-time at which the bullet with the minimal speed is shot, call it $O=O(\sigma,\tau)$. 

To do this, we show that if $\bV'$ denotes a point chosen proportionally to Lebesgue measure on $J$, then $(\bV',\bDelta)$ has the desired property with probability one. To see this, fix any $(\sigma,\tau)$ and consider the half-lines corresponding to the bullets with speeds $V_2',V_3',\dots, V_n'$ in this order. Note that the speeds $V_i'$, $2\le i\le n$ are all uniform in some small interval. For any $i\in \{2,\dots, n-1\}$, given the speeds $V_2', \dots, V_i'$ there are only a finite number of intersection points among the corresponding lines. The rays originating from $O$ intersect the half-lines $\bar \HL(V_\ell^\sigma, T^\tau_\ell)$, $2\le \sigma_\ell \le i$ at finitely many points, therefore, the line $\bar \HL(V^\sigma_k, T^\tau_k)$ with $\sigma(k)=i+1$ contains none of these with probability one. As a consequence, almost surely, no two intersections are aligned with $O=O(\sigma,\tau)$. Since the number of configurations is finite, this completes the proof.
\end{proof}
}


\begin{figure}[tbp]
\centerline{\includegraphics[width=16 cm]{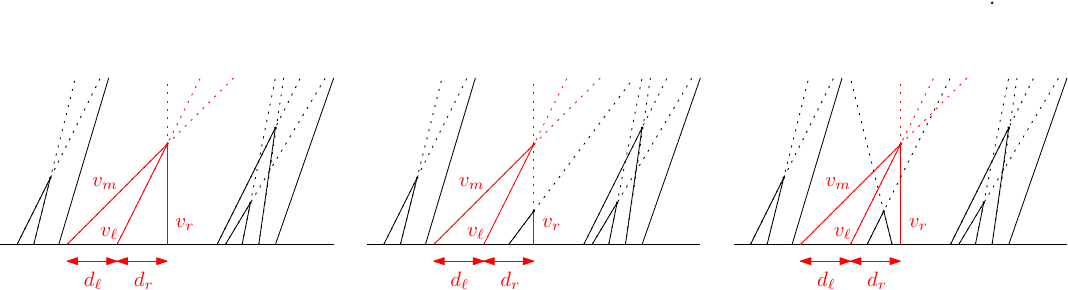}} 
\caption{\label{fig:FDN}In red a critical pattern or critical bi-triangle $(v_m,v_\ell,v_r, d_\ell,d_r)$ in different configurations. On the left, the critical pattern is realized: it is not intersected by any other trajectories. In the middle, the critical pattern is not realized, for one of the three bullets is intersected before reaching the point of the triple collision. On the right, the critical bi-triangle is realized, but not minimal. 
}  
\end{figure}

\black

\section{Invariance with respect to generic parameters}
\label{sec:dthdh}

\subsection{Statement of the invariance and consequences} 
\label{sub:from_independence_to_the_identification_of_the_law}
\label{sec:pf_theo_main}

As we already mentioned earlier, our approach consists in an induction argument. In order to better put the finger on what precisely is needed for the induction hypothesis, we state a fixed-$n$ version of the invariance principle that is one of the keys to the induction step. Note that the key assumption to guarantee that the law of the number of surviving bullets be $\bq_n$ defined in (\ref{eq:q01}--\ref{eq:q02}) is the ``invariance principle'' in \eqref{eq:local_invariance}.


\begin{lem}\label{lem:qn_heriditary}
Let $n\ge 2$. Suppose that, the following two conditions hold:
\begin{compactenum}[(i)]
  \item for every $m<n$, for every $(\bV_m,\bDelta_{m-1})\in \cG_m$, we have
  \[\bP_{\bV_m,\bDelta_{m-1}}^\mC=\bq_m\,,\]
  \item for every essentially generic parameter $(\bV_{n}, \bDelta_{n-1})$ in $\cG_{n}$, we have 
  \begin{equation}\label{eq:local_invariance}
\bP^{\mC}_{\bV_n,\bDelta_{n-1}} = \bP^{\mC}_{\bV_n^\downarrow,\bDelta_{n-1}}\,.
\end{equation}
\end{compactenum}
Then, for every $(\bV_{n},\bDelta_{n-1})\in \cG_{n}$, we have
\[\bP^{\mC}_{\bV_{n},\bDelta_{n-1}} = \bq_{n}\,.\]
\end{lem}
}
\begin{proof}
By Lemma~\ref{lem:essentia_generic}, it suffices to prove the claim for all essentially generic parameters. 
Suppose that {\em (i)} and {\em (ii)} hold, and let $(\bV_{n},\bDelta_{n-1})\in \cG_{n}$ be essentially generic, so that \eqref{eq:local_invariance} holds.
We look for a decomposition for the colliding bullet problem with $(\bV^\downarrow_n,\bDelta_{n-1})$, when $\min \bV_n^\downarrow = 0$. 
Since $\sigma$ is uniform in $\fS_n$, the element $a$ such that $\sigma_a=1$ is uniform in $\{1,\cdots,n\}$; and the speed of bullet $a$ is then 0. It happens that:
\begin{enumerate}[(1)]
  \item With probability $\frac 1n$, $a=n$. The last bullet has then speed 0 and all the others have a positive speed, so the last bullet survives. On that event, the permutation $\sigma'$ of the $n-1$ first speeds is uniform on $\fS_{n-1}$ and the last delay $\Delta^\tau_n$ is uniform among all the delays (and is independent from $\sigma'$). Thus, the number of bullets from this groups of $n-1$ that survive is distributed as a convex combination of the $\bP^{\mC}_{\bV^-,\bDelta^-}$, where $\bV^-=(V_2,V_3,\dots, V_n)$ and $\bDelta^-\in \R^{n-2}_+$ is obtained from $\bDelta_{n-1}$ by removing a uniform  component. Clearly, any such $(\bV^-,\bDelta^-)$ is generic; therefore, by induction, the convex combination is simply $\bq_{n-1}$.

  \item With probability $1-\frac 1n$, $a\in\{1,\cdots,n-1\}$. The bullet with speed 0 is shot in position $a$. The bullet that follows is shot at time $T^{\tau}_{a+1}$ and has positive speed $V^\sigma_{a+1}$: it hits the zero-speed bullet immediately at time $T^{\tau}_{a+1}$ (observe that in \eref{eq:ff}, when $v_i=0$, $T(i,j)=t_j$). In other words, the zero-speed bullet remains the barrel and is thus bound to get hit by the next bullet for it has positive speed.
\end{enumerate}
Observe now that, in the latter case (2) when $a\in \{1,2,\dots, n-1\}$: 
\begin{enumerate}[(2.i)]
  \item if $a=n-1$, then the last two bullets collide regardless of the configuration of the $n-2$ first bullets and of the delays  between them. If one conditions on the speed  $V^{\sigma}_n$ of the last bullet, and on  the delays  $(\Delta^{\tau}_{n-2},\Delta^{\tau}_{n-1})$immediately before and after the time when bullet $n-1$ is shot, then the remaining structure is entirely exchangeable, and therefore satisfies the induction hypothesis.
  \item if $a=1$, the same property holds (and the proof follows the same lines).

  \item if $1<a<n-1$, we need to condition on $(a,V^{\sigma}_{a+1},\Delta^{\tau}_{a-1},\Delta^{\tau}_{a},\Delta^{\tau}_{a+1})$. Removing bullets $a$ and $a+1$, we obtain a configuration with $n-2$ remaining bullets where the delay between bullet $a-1$ and $a+2$ is now the sum of three components $\Delta^{\tau}_{a-1}+\Delta^{\tau}_{a}+\Delta^{\tau}_{a+1}$ of $\bDelta_{n-1}$. Besides this, and the fact that two speeds $V^{\sigma}_a$ and $V^{\sigma}_{a+1}$ are fixed, the rest of the configuration is perfectly exchangeable. Now, in this case, $a$ is uniform in $\{2,\dots, n-2\}$, and since the distribution of $V^{\sigma}_{a+1}$ is uniform among the other speeds, integrating on the distribution of $a$  -- still conditioning on the other variables -- by the induction hypothesis the distribution of the surviving bullets is given by $\bq_{n-2}$. Since this is true conditionally on $(\sigma_{a+1},\tau_{a-1},\tau_a,\tau_{a+1})$ whatever these values are, the conclusion follows.
\end{enumerate}
The above decomposition implies that, for any $0\le k\le n$, we have
\[\bP_{\bV_n,\bDelta_{n-1}}^{\mC}(k) = \bP_{\bV^\downarrow_n, \bDelta_{n-1}} = \frac 1 n \bq_{n-1}(k-1) + \left(1-\frac 1 n \right) \bq_{n-2}(k)\,,\]
(with $\bq_{m}(-1)=0$ for every $m$) and it follows that $\bP_{\bV_n,\bDelta_{n-1}}^{\mC}=\bq_n$. Since $(\bV_n,\bDelta_{n-1})$ was arbitrary, this completes the proof.
\end{proof}

{Assuming Model~\ref{def:mC} follows $\bq_n$, the following proposition shows that Models~\ref{def:mA},~\ref{def:mB} and~\ref{def:mD} of Section~\ref{sec:motivation} are also governed by $\bq_n$.
\begin{pro}\label{pro:auxiliary_models}
Suppose that, for every $(\bV_n, \bDelta_{n-1})\in \cG_n$ we have $\bP_{\bV_n,\bDelta_{n-1}}^\mC=\bq_n$. Then:
\begin{compactenum}[(i)]
  \item For any laws $\mu$ (without atom) and $\nu$ (atoms allowed, except at 0), we have
\[
\bP_{n,\mu}^{\mA}=\bP_{n,\mu,\nu}^{\mB}=\bq_n.\]
  \item For any continuous (strictly) increasing function $f:\R_+\to \R_+$ with $f(0)=0$, and any $(\bV_n,\bDelta_{n-1})\in \cG_n$, we have
\[\bP^{\mAC,f}_{\bV_n,\bDelta_{n-1}} = \bq_n\,.
\]
\end{compactenum}
\end{pro} 

\begin{proof}
{\em (i)} Both  $\bP_{n,\mu}^{\mA}$ and $\bP_{n,\mu,\nu}^{\mB}$ are annealed versions of $\bP^{\mC}_{\bV,\bDelta}$. For any fixed vector $\bDelta_{n-1}$ of non-zero real numbers, $(\bV_n,\bDelta_{n-1})\in \cG_n$ almost surely if  $\bV_n$ is obtained by sorting $n$ i.i.d.\ copies of a random variable with law $\mu$. Taking the vector $\bDelta_{n-1}$ as ${\mathbf 1}=(1,1,\dots, 1)$, we obtain immediately that $\bP_{n,\mu}^{\mA}=\bq_n$. Furthermore, since $\nu$ has no atom at zero, $\min \bDelta_{n-1}>0$ with probability one when $\bDelta_{n-1}$ consists of a family of $n-1$ i.i.d. random variables with distribution $\nu$, and it follows that $\bP^{\mB}_{n,\mu,\nu}=\bq_n$.

{\em (ii)} Let $(\bV_n,\bDelta_{n-1})\in \cG_n$. Fix any continuous and (strictly) increasing $f:\R_+\to \R_+$ with $f(0)=0$. Then the map $\Phi:\R^2\to \R^2$ defined by $\Phi(x,y) =(x,f(y))$ is a one-to-one correspondence between the space-time diagrams of the model 
\[\bP^\mC_{\bV_n,\bDelta_{n-1}} 
\qquad \text{and those of}\qquad
\bP^{\mAC,f}_{\bV_n,\bDelta_{n-1}}.
\]
This one-to-one correspondence induces a one to one correspondence between the virtual collision points in both models which preserves the lexicographical order $\le_{\rm lex}$ on the plane: $(x_1,y_1) \leq_{\rm lex} (x_2,y_2) \equi \Phi(x_1,y_1)\leq_{\rm lex} \Phi(x_2,y_2)$. It follows immediately  $\Phi$ preserves the order of the collisions, and thus, the number and identities of surviving bullets.  The claim follows readily. 
\end{proof}

\subsection{Crossing a single singular point} 
\label{sub:crossing_a_single_singular_point}

We now state the main element of our strategy:
\begin{lem}The assumption $(ii)$ of Lemma~\ref{lem:qn_heriditary} holds: for every essentially generic parameter $(\bV_{n}, \bDelta_{n-1})$ in $\cG_{n}$, we have $\bP^{\mC}_{\bV_n,\bDelta_{n-1}} = \bP^{\mC}_{\bV_n^\downarrow,\bDelta_{n-1}}\,$.
\end{lem}
The proof of this lemma will be decomposed, but it will somehow last until the end of Section \ref{sec:restrictions}. 
The proof roughly consists in showing that, as we continuously lower the minimum speed of an essentially generic parameter $(\bV,\bDelta)$, the probability distribution $\bP^{\mC}_{\bV,\bDelta}$ remains unchanged.} This is only partially true, since as we lower the minimum speed, we may encounter singular parameters, for which the colliding bullet problem is not even well-defined. Fix $(\bV,\bDelta)$ an essential generic parameter in $\cG_n$. For $\lambda\in [0,1]$, the parameter 
\[(\bV_\lambda, \bDelta):=((1-\lambda) \bV + \lambda \bV^\downarrow,\bDelta)\] is either a generic or a simple singular parameter; furthermore, there are at most finitely many values $0<\lambda_1<\lambda_2<\dots < \lambda_k<1$ for which $(\bV_\lambda,\bDelta)$ is singular. By Lemmas~\ref{lem:tcs} and~\ref{lem:TCS_open}, the map
\begin{equation}\label{eq:map_lowering_speed}
\lambda \mapsto \bP_{\bV_\lambda,\bDelta}^\mC
\end{equation}
is constant on each of the intervals $[0,\lambda_1)$, $(\lambda_i,\lambda_{i+1})$, $1\le i<k$ and $(\lambda_k,1]$. As a consequence, proving
 that the hypothesis of  Lemma \ref{lem:qn_heriditary} holds 
reduces to showing that, for $1\le i\le k$, the left and right limits at $\lambda_i$ agree. 
Note that, by construction, for every $1\le i\le k$, the parameter $(\bV_{\lambda_i},\bDelta)$ is simple singular, and every critical pattern for $(\bV_{\lambda_i},\bDelta)$ involves the minimum speed $\lambda_i V_1$. In the following, we call such a singular parameter \emph{honest}. This latter fact is crucial for the arguments to come.
In other words, up to a change of variables, one is lead to studying the difference between $\bP^{\mC}_{\bV+,\bDelta}$ and $\bP^\mC_{\bV-,\bDelta}$ where $(\bV,\bDelta)$ is an honest singular parameter, and
\[\bP^\mC_{\bV+,\bDelta} := \lim_{\lambda \to 0-} \bP^\mC_{\bV_\lambda, \bDelta}
\qquad \text{and}\qquad 
\bP^\mC_{\bV-,\bDelta} := \lim_{\lambda \to 0+} \bP^\mC_{\bV_\lambda, \bDelta}\,.\]
In the following, we refer to this as ``crossing the singular point $(\bV,\bDelta)$''. Since the permutations $(\sigma,\tau)$ are uniformly random, proving that the two laws $\bP^\mC_{\bV+,\bDelta}$ and $\bP^\mC_{\bV-,\bDelta}$ agree consists in verifying that, for every $k\ge 0$, the number of the configurations for which $k$ bullets survive agree for both parameters. In the following, we often refer to these two limit colliding bullets problem as $(\bV-,\bDelta)$ and $(\bV+,\bDelta)$, and what we mean here is that the pair $(\bV_-, \bDelta)$, $(\bV_-,\bDelta)$ refers to any choice of $(\bV_{-\lambda},\bDelta)$, $(\bV_{\lambda},\bDelta)$ with $\lambda>0$ small enough such that $(\bV,\bDelta)$ is the only singular parameter among $(\bV_\mu,\bDelta)$ with $\mu\in [-\lambda,\lambda]$, the actual choice being in fact irrelevant.

There are some natural classes of configurations for which the numbers ``obviously'' agree. We now expose some of those classes in order to better focus the remainder of the proof to the classes of configurations for which some genuine work is needed. 

\medskip
\noindent (i) \textsc{Only critical configurations matter.} 
 In  words:  the contribution to $\bP^\mC_{\bV+,\bDelta}$ and to $\bP^\mC_{\bV-,\bDelta}$ of the configurations which do not contain any critical pattern is the same. Formally,  we say that a configuration $(\sigma,\tau)$ is \emph{critical} for $(\bV,\bDelta)$ if there exists a (necessarily unique) $(\sigma,\tau)$-critical pattern with respect to $(\bV,\bDelta)$. We denote by $\cal C(\bV,\bDelta)=\cup_\pi \cal C_\pi(\bV,\bDelta)$ the set of configurations that are critical for $(\bV,\bDelta)$. If $(\sigma,\tau)\not\in \cal C(\bV,\bDelta)$ then $\Gamma_{\bV+,\bDelta}(\sigma,\tau, \cdot, \cdot,\cdot ) = \Gamma_{\bV-,\bDelta}(\sigma,\tau, \cdot, \cdot,\cdot )$. Therefore, for every $k\ge 0$, the numbers of non-critical configurations in which $k$ bullets survive are identical for $(\bV-,\bDelta)$ and $(\bV+,\bDelta)$. 

\medskip
We now look further at the critical configurations in $\cal C(\bV,\bDelta)$. Recall the notion of a realized critical pattern and of a minimal critical pattern defined in Section~\ref{ssec:singular-critical}.

\medskip
\noindent (ii) \textsc{Only configurations with realized critical patterns matter.} In words: 
the contribution to $\bP^\mC_{\bV+,\bDelta}$ and to $\bP^\mC_{\bV-,\bDelta}$ of the configurations without any realized critical pattern is the same. Formally, consider $\pi=(\min \bV,v_\ell,v_r,d_\ell,d_r)$ a critical pattern for $(\bV,\bDelta)$. Fix a configuration $(\sigma,\tau)\in \cal C_\pi \setminus \cal R_\pi$ for which the pattern is not realized. Then, by definition, at least one of the bullets with speeds $\min \bV$, $v_\ell$ or $v_r$ does not survive until the time where the triple-collision is supposed to occur. As a consequence, the space-time diagrams of $(\bV-,\bDelta)$ and $(\bV+,\bDelta)$ corresponding to $(\sigma,\tau)$ are identical, and therefore, the number of surviving bullets is the same in both situations. It follows that it suffices to consider the configurations $(\sigma,\tau)\in \cup_\pi \cal R_\pi$ in which the critical pattern is realized. 

\medskip 
\noindent (iii) \textsc{It suffices to consider minimal critical patterns.} In words: 
the contribution to $\bP^\mC_{\bV+,\bDelta}$ and to $\bP^\mC_{\bV-,\bDelta}$ of the configurations with a realized critical pattern which is not minimal in the sense that it is not formed by three consecutive bullets can be reduced to a similar structure on a bullet problem with less than $n$ bullets, and  can thus be treated by induction. Formally,
let $\pi=(\min \bV, v_\ell, v_r, d_\ell, d_r)$ be a critical pattern that is not minimal, that is at least one of $d_\ell$ or $d_r$ is the sum of more than one component of $\bDelta$. Consider now $\cal R_\pi$, and suppose that it is not empty (this requires, for instance, that both $d_\ell$ and $d_r$ are sums of an odd number of elementary delays to ensure that the bullets shot ``within the critical bi-triangle'' pairwise annihilate). 

Fix $(\sigma,\tau)\in \cal R_\pi$. Let $T^\tau_i$, $T^\tau_{j}$ and $T^\tau_{k}$, with $1\le i<j<k\le n$ denote respectively the times at which the bullets with speeds $V_1=\min \bV$, $v_\ell$ and $v_r$ are shot. The configuration $(\sigma,\tau)$ may be decomposed into two parts: an ``outer'' configuration of speeds and delays on $[0,T^\tau_i] \cup [T^\tau_k, T^\tau_n]$ and another ``inner'' configuration of speeds and delays on $[T^\tau_i,T^\tau_k]$. Of course, these configurations are constrained by durations of the intervals, and the fact that both the bullets shot from $[0,T^\tau_i) \cup (T^\tau_k,T^\tau_n]$ and from $(T^\tau_i,T^\tau_k) \setminus \{T^\tau_j\}$ should avoid the trajectories of the bullets with speeds $V_1$, $v_\ell$ and $v_r$ before the triple-collision.

We can decompose the configuration $(\sigma,\tau)$ as follows:
\begin{itemize}
  \item \emph{outer configuration:} Let ${\cal I}:=\{\sigma_p: T^\tau_p<T^\tau_i \text{ or }T^\tau_p> T^\tau_k\}$ and ${\cal J}:=\{\tau_p: p< i \text{ or } p\ge k\}$. Then, $|{\cal I}|=i+n-k-1$ and $|{\cal J}|=i-1+n-k$. Let $\bV^\bullet$ denote the increasing reordering of $\{V_p: p\in {\cal I}\}\cup \{V_1,v_\ell,v_r\}$, and $\bDelta^\bullet$ the increasing reordering of $\{\Delta_p: p \in {\cal J}\} \cup \{d_\ell, d_r\}$. Then, $(\bV^\bullet, \bDelta^\bullet) \in \Theta_{n+i-k+2}$. 

  \item \emph{inner configuration:} Let $\bV^\circ$ and $\bDelta^\circ$ be the increasing reorderings of $\{V_p: p\not \in {\cal I}\}$ and $\{\bDelta_p: p\not \in {\cal J}\}$, respectively. Then $(\bV^\circ,\bDelta^\circ)\in \Theta_{k+1-i}$. 
\end{itemize}
The configuration $(\sigma,\tau)$ then corresponds to a pair of configurations, say $(\sigma^\bullet,\tau^\bullet)$ and $(\sigma^\circ,\tau^\circ)$ for $(\bV^\bullet, \bDelta^\bullet)$ and $(\bV^\circ, \bDelta^\circ)$, respectively. Both $(\sigma^\bullet,\tau^\bullet)$ and $(\sigma^\circ,\tau^\circ)$ contain the critical pattern $\pi$, and it is realized; furthermore, $\pi$ is minimally critical for $(\bV^\bullet, \bDelta^\bullet)$. We emphasize the fact that, while a given configuration gives rise to a single pair of parameters $(\bV^\bullet,\bDelta^\bullet)$, $(\bV^\circ,\bDelta^\circ)$, in general, there may be more than one pair of parameters when one considers all the configurations $(\sigma,\tau)\in {\cal R}_\pi(\bV,\bDelta)$. 

The fact that $\pi$ is realized yields a kind of decoupling between the configurations $(\sigma^\bullet,\tau^\bullet)$ and $(\sigma^\circ,\tau^\circ)$: given any achievable pair of parameters $(\bV^\bullet,\bDelta^\bullet)$, $(\bV^\circ, \bDelta^\circ)$, and all the configurations $(\sigma^\bullet,\tau^\bullet)$ for which $\pi$ is realized, every configuration of $(\sigma^\circ,\tau^\circ)$ for which $\pi$ is also realized is compatible with $(\sigma^\bullet,\tau^\bullet)$. Therefore, the number of configurations $(\sigma,\tau)$ for which $\pi$ is realized and the decomposition in pair of parameters is $(\bV^\bullet,\bDelta^\bullet)$, $(\bV^\circ, \bDelta^\circ)$ is the product $\#\{(\sigma^\bullet,\tau^\bullet) \in {\cal R}_\pi\} \times \#\{(\sigma^\circ,\tau^\circ) \in {\cal R}_\pi\}$. Furthermore, since the only bullets that may survive are those with speed in $\bV^\bullet$, we have the refined relation:
\begin{align*}
& \# \{(\sigma,\tau)\in {\cal R}_\pi(\bV,\bDelta): |\cS_{\bV,\bDelta}(\sigma,\tau)| = k\}\\
& = \sum \#\{(\sigma^\bullet,\tau^\bullet)\in {\cal R}_\pi(\bV^\bullet,\bDelta^\bullet): |\cS_{\bV^\bullet,\bDelta^\bullet}(\sigma^\bullet,\tau^\bullet)|=k\} 
\times \#\{(\sigma^\circ,\tau^\circ)\in {\cal R}_\pi(\bV^\circ,\bDelta^\circ)\}\,,
\end{align*}
where the sum in the right-hand side extends over achievable pairs of parameters $(\bV^\bullet,\bDelta^\bullet)$, $(\bV^\circ,\bDelta^\circ)$. In other words, if we proceed by induction on $n$, the fact that the numbers of configurations where some non-minimal critical pattern is realized agree is a direct consequence of the fact that the numbers agree for all minimal critical patterns in a colliding bullets problem of smaller size.

\medskip
We can now state the result of the arguments above: 

\begin{pro}\label{pro:real_pattern}Let $(\bV,\bDelta)$ be an honest simple singular parameter. Moreover, let $\pi=(\min \bV, v_\ell,v_r,d_\ell,d_r)$ be a minimal critical pattern for $(\bV,\bDelta)$. If the distribution of the number of surviving bullets in $(\bV-,\bDelta)$ and $(\bV+,\bDelta)$ agree when we restrict the count to configurations in $\cal R_\pi$, then the distribution is preserved over $\fS_n\times \fS_{n-1}$.
\end{pro}


\subsection{On the need for keeping track of constraints}

The arguments of the previous section, summarized in Proposition~\ref{pro:real_pattern}, imply that we may focus on honest simple singular parameters, and on configurations in which a minimal critical pattern is realized. The minimality of the pattern and the fact that the minimal speed $\min \bV$ is involved in the pattern imply that the constraints that $\pi$ be realized in $(\sigma,\tau)$ reduces to the fact that no bullet hits the critical bitriangle from the right. 

Now, a glance at Figure~\ref{fig:FDN2} suffices to note that, for a configuration in $\cal R_\pi$, the sole effect of passing from $\bV+$ to $\bV-$ is to release the line with speed $v_r$, and replace the collision between the bullets with speeds $v_\ell$ and $v_r$ by the collision between the bullets with speeds $\min \bV$ and $v_\ell$. Furthermore, in view of Fig.~\ref{fig:FDN2}, for the configurations containing the minimal and realized bi-triangle $\pi$, the two colliding bullets -- those with speeds $v_\ell$ and $v_r$ before slowing down in $(\bV+,\bDelta)$ and those with speeds $\min \bV$ and $v_\ell$ in $(\bV-,\bDelta)$ -- do not contribute to the number of surviving bullets. In spite of this, we cannot just suppress them:
\begin{itemize}
  \item \textsc{it would modify the colliding problem.} For example, consider the case of $\bV+$. Since the bullets with speeds $v_\ell$ and $v_r$ collide, the delays $d_{\ell}$ and $d_r$ merge. If we were to remove the bullets with speeds $v_\ell$ and $v_r$, there would be no bullet shot at time $d_\ell+d_r$ after the bullet $\min \bV$. This would imply that the delay between the bullet with minimal speed and the next one (if any), is the sum of three elementary intervals. 
  \item \textsc{Constraints persist.} More importantly, the removal in the space-time diagram of the bullets with speeds $v_\ell$ and $v_r$ does not remove the constraints they were holding: we restricted ourselves to configurations where the critical pattern is realized; this restriction \emph{a priori} imposes restrictions on the configurations that may be obtained when removing the bullets with speeds $v_\ell$ and $v_r$.
\end{itemize}

For all the above reasons, we are led to consider more general models that allow to keep track of the constraints imposed by the assumptions along the course of the induction argument. This is developed in the next section, where we also complete the proof of Theorem~\ref{theo:main}.

\begin{figure}[tbp]
\centering
\includegraphics[width=16 cm]{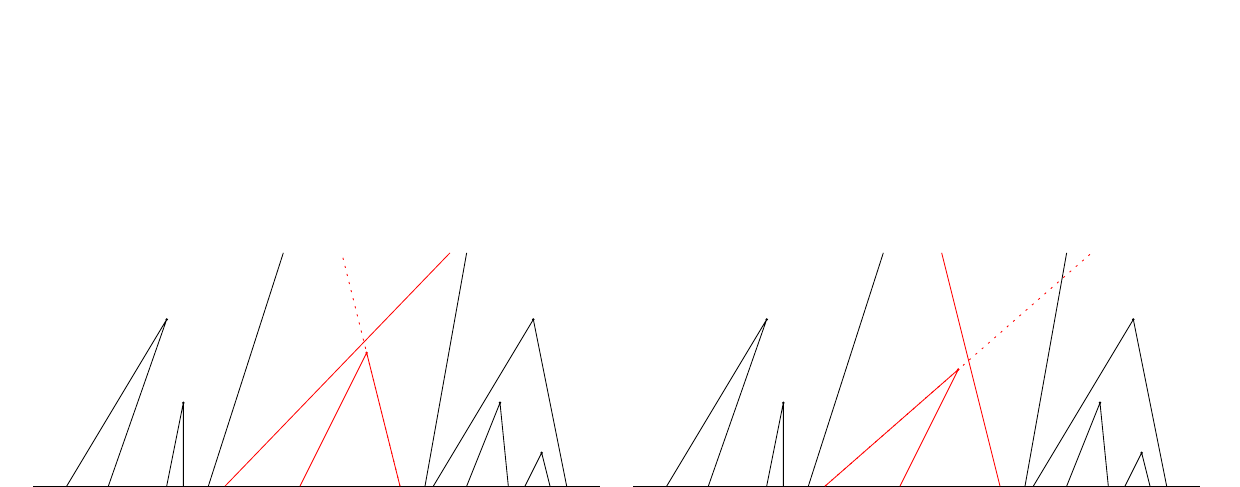}
\caption{\label{fig:FDN2} The evolution of the space-time diagram when crossing a singular parameter $(\bV,\bDelta)$ for a configuration where the critical pattern is realized: The effect of reducing the speed of the bullet forming the left-side of the bi-triangle till it crosses the singular value of the parameter is to switch the pair of bullets colliding in the vicinity of the triple point, and hence to also switch the bullet that survives. }
\end{figure}

\section{Counting configurations with general restrictions}
\label{sec:restrictions}

\subsection{Two combinatorial models with constraints}


The considerations of the previous section motivate the introduction of two colliding bullets models with restrictions, which 
generalize the initial colliding bullets problem. It is the relationship between these two models that allows to compare the distributions $\bP^\mC_{\bV+,\bDelta}$ and $\bP^\mC_{\bV-,\bDelta}$ of the number of surviving bullets when crossing a minimal singular parameter. 
Both models are very similar to the initial colliding bullets problem, except that: there are a distinguished delay $\Delta^{\star}$, a distinguished speed $V_r$ and a distance $s$ which, together with the minimal speed $V_{\min}$ enter into play to constrain the set of configurations that are allowed. 

The vectors of speeds and delays are now denoted by $\bV_n=(V_1,\cdots,V_{n-2},V_{\min},V_r)$ and $\bDelta_{n-1}=(\Delta_1,\cdots,\Delta_{n-2},\Delta^{\star})$, respectively\footnote{For the sake of simplicity, $\bV$ is not a sorted vector anymore.}; we also enforce that $V_{\min}$ is the minimal speed of $\bV_n$. Hence, a configuration is now a pair of permutations $(\sigma,\tau)\in \fS_{n-2}\times \fS_{n-1}$. Given the permutation $\tau\in \fS_{n-1}$ of the delays, we define a sequence of times $(T^\tau_i)_{1\le i\le n}$, by 
\ben
T^{\tau}_i= \Delta_{\tau_1}+\cdots+\Delta_{\tau_{i-1}},
\een 
where for convenience we have written $\Delta_{n-1}:=\Delta^\star$. Two of these times are distinguished and correspond to the beginning and to the end of the interval corresponding to the distinguished delay $\Delta^{\star}$; we denote them by $\bar T^\tau_\scl$ and $\bar T^\tau_\scr$, with the constraint that $\bar T^\tau_\scr - \bar T^\tau_\scl=\Delta^\star$. The $n-2$ remaining non-distinguished times are distinct and come with a natural ordering, and we denote them by $\bar T^\tau_i$, $1\le i\le n-2$.

The speeds are then assigned to the times as follows: the speeds $V_{\min}$ and $V_r$ are assigned to the distinguished times $\bar T^\tau_{\scl}$ and $\bar T^\tau_{\scr}$ in this order.  The permutation $\sigma\in \fS_{n-2}$ then determines to which non-distinguished time is assigned each one of the $n-2$ non-distinguished speeds. We let 
\ben
H=H(\bV_n,\bDelta_{n-1})
\een denote the distance between the horizontal axis and the point $\bar \HL(V_{\min},0) \cap \bar \HL(V_r,\Delta^\star)$, which exists and lies above the axis since $V_{\min} < V_r$ (see as Figure~\ref{fig:tukyi}.)

The two new models, which, in passing, are combinatorial models, are parametrized by $(\bV_n,\bDelta_{n-1})\in \cG_n$, a real number $s\in [0,H]$, and a set $A$ which will take in the sequel one of the three values:  $\{0\}$, $\mathbb{Z}^+$ or $\mathbb{Z}^+\setminus \{0\}$. In both models, there is a special segment 
\[S:=\bar \HL(V_r,\bar T^\tau_{\scr})~\bigcap~[0,\infty)\times [0,s].\]
The restriction will come from the number of bullets whose true trajectory hits the segment $S$ in the space-time diagram. In both models, there is only one bullet that is shot from one of the extremities of the distinguished interval corresponding to $\Delta^\star$, and the name of the model refers to whether it is shot from the \emph{left} or the \emph{right} end point. We phrase the models combinatorially, but one should keep in mind that since we will counting configurations, there is an underlying uniform measure on the $(\sigma,\tau)\in {\fS}_{n-2}\times {\fS}_{n-1}$.

\begin{model}{-- \sc Left model with restriction $\LR(\bV_n,\bDelta_{n-1},s, \cdot,\cdot)$.} \label{def:LR}
For $(\sigma,\tau)\in{\fS}_{n-2}\times {\fS}_{n-1}$:
\begin{compactitem}[\textbullet]
  \item Shoot the bullet with speed $V_{\sigma_j}$ at time $\bar{T}^{\tau}_j$ for $1\le j\le n-2$; this gives rise to the virtual trajectories $\bar \HL(V^\sigma_j, \bar T^\tau_j)$, $1\le j\le n-2$.
  \item Shoot the bullet with minimal speed $V_{\min}$ at time $\bar{T}^{\tau}_\scl$; this corresponds to the virtual trajectory $\bar \HL(V_{\min},\bar{T}^{\tau}_\scl)$.
  \item No bullet is shot at time $\bar T^\tau_{\scr}$. 
\end{compactitem}
 \medskip
For $k\ge 0$, we denote by $\LR(\bV_n,\bDelta_{n-1},s, A,k)$ the set of configurations $(\sigma,\tau)$ such that in the (true) space-time diagram of the $n-1$ bullets, the number of bullets whose true trajectory crosses the segment $S$ belongs to $A$, and $k$ bullets eventually survive. 
\end{model}

\begin{figure}[tb]
\centering 
\includegraphics[width=16cm]{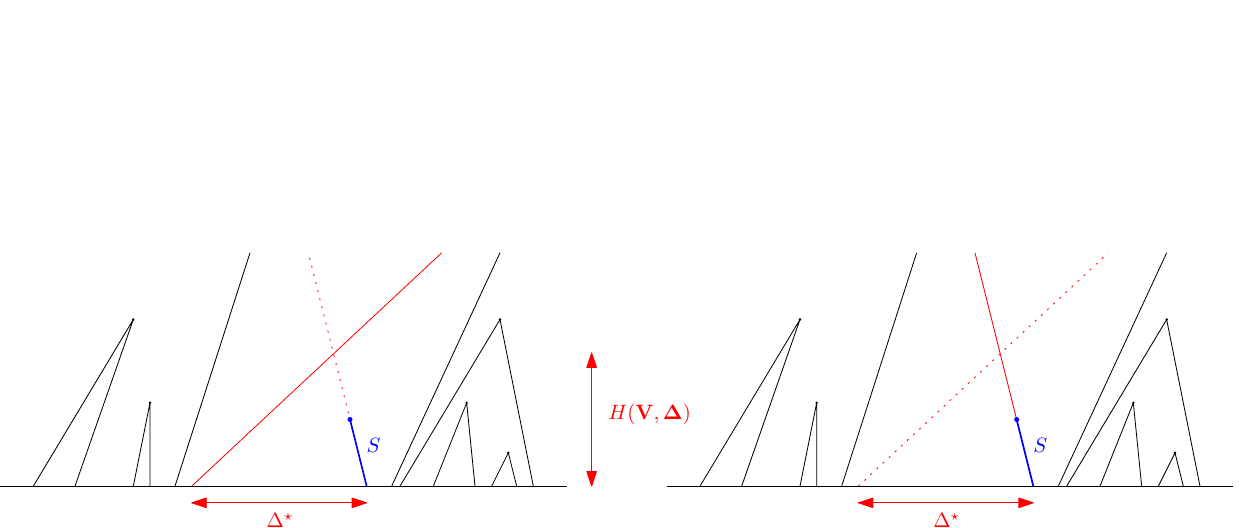}
\caption{\label{fig:tukyi} 
The same configuration is represented for the two constrained models (Models~\ref{def:LR} and~\ref{def:RR}). In red, the only lines corresponding to the distinguished speeds, forming what we will refer to as the \emph{special triangle}; in blue the special segment $S$ used to restrict the configurations that are involved. The left side of the special triangle is denoted by $LS$, and the right side by $RS$.}
\end{figure}

\begin{model}{-- \sc Right model with restriction $\RR(\bV_n,\bDelta_{n-1},s, \cdot,\cdot)$.}\label{def:RR}
For $(\sigma,\tau)\in{\fS}_{n-2}\times {\fS}_{n-1}$:
\begin{compactitem}[\textbullet]
  \item Shoot the bullet with speed $V_{\sigma_j}$ at time $\bar{T}^{\tau}_j$ for $j$ from $1$ to $n-2$; this gives rise to the virtual trajectories $\bar \HL(V^\sigma_j, \bar T^\tau_j)$, $1\le j\le n-2$.
  \item No bullet is shot at time $\bar T^\tau_{\scl}$.
  \item Shoot the bullet with speed $V_r$ at time $\bar T^\tau_{\scr}$; this corresponds to the virtual trajectory $\bar \HL(V_r, \bar T^\tau_{\scr})$.
\end{compactitem}
\medskip
For $k\ge 0$, Denote by $\RR(\bV_n,\bDelta_{n-1},s, A,k)$ the set of configurations $(\sigma,\tau)$ such that in the (true) space-time diagram, the number of bullets that cross the segment $S$ belongs to $A$, and $k$ bullets survive. (\emph{By convention, the bullet with speed} $V_r$ \emph{does not cross} $S$.)
\end{model}
{ 

The relationship between the two models when $s=H$ and $A=\{0\}$ is precisely the one between the two versions of the original colliding bullet problem taken at two different parameters in the vicinity of a singular parameter: the consequence of a modification of the minimal speed is that the bullet which survives the ``near triple collision'' switches from the one with speed $V_{\min}$ to the one with speed $V_r$ (or vice versa). (Compare Figures~\ref{fig:FDN2} and \ref{fig:tukyi}.) The other values for $s$ and $A$ are used for the proof.

Let $n\ge 2$ be a natural number and consider the three following properties; recall the sequence of probability distributions $\bq_n$ defined in \eqref{eq:q01} and \eqref{eq:q02}.
\begin{align*}
\setlength{\jot}{30pt}
{\cal P}_n^{(1)}:& 
\quad~ \left\{~\,\text{for any generic parameter } (\bV_n,\bDelta_{n-1}) \in \cG_n \text{ we have } \bP_{\bV_n,\bDelta_{n-1}}^{\mC} = \bq_n ~\,
\right\}\\[10pt]
{\cal P}_n^{(2)}: &
\quad~\left\{
\begin{array}{c}
~~\text{for the set }A \text{ being either }\{0\}, \mathbb{Z}_+, \text{ or }\mathbb{Z}_+\setminus \{0\} \text{ we have }\\
\text{ for all } (\bV_n,\bDelta_{n-1})\in \cG_n, \text{for all } s\leq H(\bV_n,\bDelta_{n-1}), \text{ and for all } k\ge 0\!\\ 
\l|\LR(\bV_n, \bDelta_{n-1},s, A,k)\r|=\l|\RR(\bV_n, \bDelta_{n-1},s, A,k)\r| \quad 
\end{array}
\right\}\\[10pt]
{\cal P}_n^{(3)}:&
\quad~\left\{
\begin{array}{c}
\text{there exists a map }g_n:\mathbb Z_+ \to \mathbb Z_+ \text{ such that} \\
\text{for any }(\bV_n,\bDelta_{n-1})\in \cG_n \text{ and for all }k\ge 0 \text{ we have }\\
\quad~ |\LR(\bV_n, \bDelta_{n-1},0, \{0\},k)|=|\RR(\bV_n, \bDelta_{n-1},0, \{0\},k)| = g_n(k)\quad~~\,
\end{array}
\right\}\,.
\end{align*}

\begin{rem}(i) In view of Proposition~\ref{pro:auxiliary_models}, the part that is currently missing to complete the proof of Theorem~\ref{theo:main} is precisely the fact that ${\cal P}_n^{(1)}$ holds for every $n$.

\noindent (ii) We emphasize the fact that, in ${\cal P}_n^{(2)}$, it is not true that the two cardinalities
\[|\LR(\bV_n,\bDelta_{n-1}, s, A, k)| \qquad \text{and} \qquad |\RR(\bV_n,\bDelta_{n-1}, s, A, k)|\]
are independent of $(\bV_n,\bDelta_{n-1})\in \cG_n$.
\end{rem}
}
\begin{pro}\label{theo:main2}
For any $n\ge 2$, the properties ${\cal P}_n^{(1)}$, ${\cal P}_n^{(2)}$ and ${\cal P}_n^{(3)}$ all hold.
\end{pro}
We will prove Proposition~\ref{theo:main2} by induction. We could not find a argument that would proceed by proving ${\cal P}_n^{(1)}$, ${\cal P}_n^{(2)}$ and ${\cal P}_n^{(3)}$ each separately, and it seems that one has to treat the bundle 
\ben
{\cal P}_n= \big \{\text{ all three properties }{\cal P}_n^{(1)}, {\cal P}_n^{(2)} \textrm { and }{\cal P}_n^{(3)} \textrm{ hold } \big\}
\een
in a single induction argument. This is the main reason why the following proof is slightly intricate.

First, one easily verifies that ${\cal P}_1$ and ${\cal P}_2$ both hold: this can be checked by inspecting the cases: 0 or 1 bullet is fired in all these models, except in ${\cal P}_2^{(1)}$; this latter case with two bullets is simple enough to be checked easily. Now, the induction step necessary to prove Proposition~\ref{theo:main2} follows directly from Lemmas~\ref{lem:ind3}, \ref{lem:ind1} and \ref{lem:ind2} below:

\begin{lem}\label{lem:ind3}Let $n\ge 2$. If ${\cal P}_n$ holds, then ${\cal P}_{n+1}^{(3)}$ holds.
\end{lem}

\begin{lem}\label{lem:ind1}Let $n\ge 2$. If ${\cal P}_m$ holds for all $m\le n$, then ${\cal P}_{n+1}^{(1)}$ holds.
\end{lem}

\begin{lem}\label{lem:ind2}Let $n\ge 2$. If ${\cal P}_n$ holds, then ${\cal P}_{n+1}^{(2)}$ holds.
\end{lem}

Observe that Lemma~\ref{lem:ind1} is what makes the link between the original bullet colliding problem and the models with constraints described above precise.  The proofs of these three Lemmas are presented in the next Sections..

\subsection{Proof of Lemma~\ref{lem:ind1}} 
\label{sub:proof_of_lemma1}
{
Since $\cP_{m}$ holds for $m\le n$, by Lemma~\ref{lem:qn_heriditary}, it suffices to prove that \eqref{eq:local_invariance} holds for all $(\bV_{n+1}, \bDelta_n)\in \cG_{n+1}$ that is essentially generic. Then, by the arguments in Section~\ref{sub:crossing_a_single_singular_point} this reduces to comparing the law of the number of surviving bullets for $(\bV^\star_{n+1}+,\bDelta_n)$ and $(\bV^\star_{n+1}-,\bDelta_n)$, where $(\bV_{n+1}^\star, \bDelta_n)$ is a simple singular parameter $\Theta_{n+1}$.

By Proposition~\ref{pro:real_pattern}, in order to compare $(\bV^\star_{n+1}+,\bDelta_n)$ and $(\bV^\star_{n+1}-,\bDelta_n)$ it suffices to consider configurations for which there is some minimal critical pattern with respect to $(\bV^\star_{n+1},\bDelta_n)$ that is realised and involves the minimum speed. 

There is a unique critical pattern for $(\bV^\star_{n+1},\bDelta_n)$ and it is of the form $\pi=(V_1, V_\ell, V_r, d_\ell, d_r)$ where both $d_\ell$ and $d_r$ are components of $\bDelta_n$. For any configuration $(\sigma,\tau)$ where $\pi$ is realized, }
\begin{itemize}
  \item in $(\bV^\star_{n+1}+,\bDelta_n)$ the bullets with speeds $V_{\ell}$ and $V_r$ collide;
  \item in $(\bV^\star_{n+1}-,\bDelta_n)$ the bullets with speeds $V_1=\min \bV_{n+1}^\star$ and $V_\ell$ collide.
\end{itemize}
Note that in both models, $V_\ell$ does not play a role after the triple-collision, and it will be suppressed from the parameter in the following. More precisely, writing $\Delta^\star=d_\ell + d_r$, and reorganizing the components of the parameter into
\[
\begin{array}{ll}
\bV_n' = (V_2,\dots, V_{n-1}, V_1, V_r) \in \R^n_+ & \text{where the speed }V_\ell \text{ has been removed},\\
\bDelta_{n-1}' = (\Delta_1,\dots, \Delta_n, \Delta^\star) \in \R^{n-1}_+ & \text{since $d_\ell$ and $d_r$ have been merged into }\Delta^\star
\end{array}
\]
we are precisely led to proving that 
\begin{equation}\label{eq:remove_crit_triangle}
\l|\LR\l(\bV'_{n},\bDelta_{n-1}',s,\{0\},k\r)\r|= \l|\RR\l(\bV_n', \bDelta_{n-1}',s, \{0\},k\r)\r|\qquad \text{for all } k\ge 0,  
\end{equation}
for $s=H(\bV_n,\bDelta_{n-1})$. 
Since  we assumed $\cP^{(2)}_n$, \eqref{eq:remove_crit_triangle} holds and, in turn, 
\[\bP^\mC_{\bV^\star_{n+1}+,\bDelta_n} = \bP^\mC_{\bV^\star_{n+1}-,\bDelta_n}\,,\]
which completes the proof.

\subsection{Proof of Lemma~\ref{lem:ind2}} 
\label{sub:proof_of_lemma2}

First note that it suffices to establish the formula for $A=\Z_+$ and for $\Z_{+}\setminus\{0\}$ since one can then recover the case $A=\{0\}$ by a simple difference. Suppose that $\cP_m$ hold for all $m\le n$ and fix $(\bV_{n+1},\bDelta_n)\in \cG_{n+1}$.

\medskip

\noindent \emph{The case $A=\Z_+$.} In this case, there is no constraint on the number of trajectories intersecting the special segment $S$. As a consequence, the previous arguments (that say that the case $A=\Z_+$ is equivalent to the case $s=0$) show that $\cP_{n+1}^{(3)}$ implies the desired property, and Lemma~\ref{lem:ind3} proves that $\cP_n$ implies $\cP_{n+1}^{(3)}$ (and is proved independently of the current lemma later on).

\medskip
\noindent \emph{The case $A=\Z_{+}\setminus \{0\}$}. 
\begin{figure}[tbp]
\centerline{\includegraphics{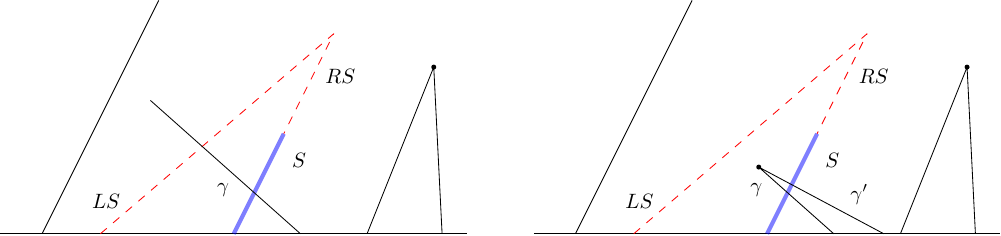}} 
\caption{\label{fig:PLO}Illustration for the proof of Lemma~\ref{lem:ind2}: Either $\gamma$ touches $LS$ or it is intersected before by some other line. (The lines $\gamma$ and $\gamma'$ have been chosen with negative slopes only for the sake of clarity of the representation.)}
\end{figure} 
In each of the configurations involved, some of the $\bar \HL_i$ hit the segment $S$ from the right. Consider those for which the lowest intersection point in $S$ with these half lines is reached by a given half line $\gamma$. Now, denote by $RS$ and $LS$ the right and left sides of the triangle formed by the line segments with speed $V_{\min}$ and $V_r$. There are two cases (represented in Figure~\ref{fig:PLO}) depending on whether $\gamma$ is intersected by some other half-line before reaching $LS$ or not:
\begin{enumerate}[(a)]
  \item \emph{if $\gamma$ touches $RS$ and $LS$ without being intersected}: in this case, we prove directly (without the induction hypothesis) that, for all $k\ge 0$, we actually have equality of the two sets
  \begin{equation}\label{eq:gamma}
  \LR(\bV_{n+1},\bDelta_{n},s,\Z_+\setminus\{0\},k,\gamma) = \RR(\bV_{n+1},\bDelta_{n},s,\Z_+\setminus\{0\},k,\gamma)\,,
  \end{equation}
where we added the entry $\gamma$ to $\LR$ and $\RR$ to denote the set of configurations corresponding to this situation.
  For any configuration $(\sigma,\tau)\in \LR(\bV_n,\bDelta_{n-1},s,\Z_+\setminus\{0\},k,\gamma)$, 
  it is clear that $(\sigma,\tau)\in \cup_{\ell\ge 0}\RR(\bV_n,\bDelta_{n-1}, s, \Z_+\setminus \{0\},\ell)$, the only thing to prove is that taking $(\sigma,\tau)$ as a configuration of $\RR$, there are precisely $k$ surviving bullets. To see this, observe first that the line $\gamma$ touches $RS$ and $LS$ without being hit in the LR model, and hits $RS$ in the $\RR$ model. Consider the triangle $T$ formed by the horizontal line, the direction $\gamma$ and $LS$. Now, the bullet whose trajectory follows $\gamma$ (see Figure~\ref{fig:firstpart}):
\begin{itemize}
  \item hits the bullet with minimum speed whose trajectory follows $LS$ in $\LR$, and
  \item hits the bullet with speed $V_r$ whose trajectory follows $RS$ in $\RR$.
\end{itemize}
Therefore, in each case, the two corresponding bullets collide. Now, the rest of the configuration is unaffected by any of the collisions, because everything in both cases, happens inside the triangle $T$, which is not intersected in both models. As a consequence, if there are $k$ surviving bullets in $\LR$, there are also $k$ surviving bullets in $\RR$. One easily sees that the argument actually proves the equality of the two sets in \eqref{eq:gamma}\footnote{We note here that we do not need to assume that the line $\gamma$ is the line that immediately come after $RS$. There may well be some bullets whose (real) trajectories are trapped in the triangle formed by $RS$ and $\gamma$; these are precisely the same in $\LR$ and $\RR$ because of our assumption that $\gamma$ is the lowest line hitting $RS$.}. 

  \begin{figure}[htbp]
\centering
\includegraphics[width=16cm]{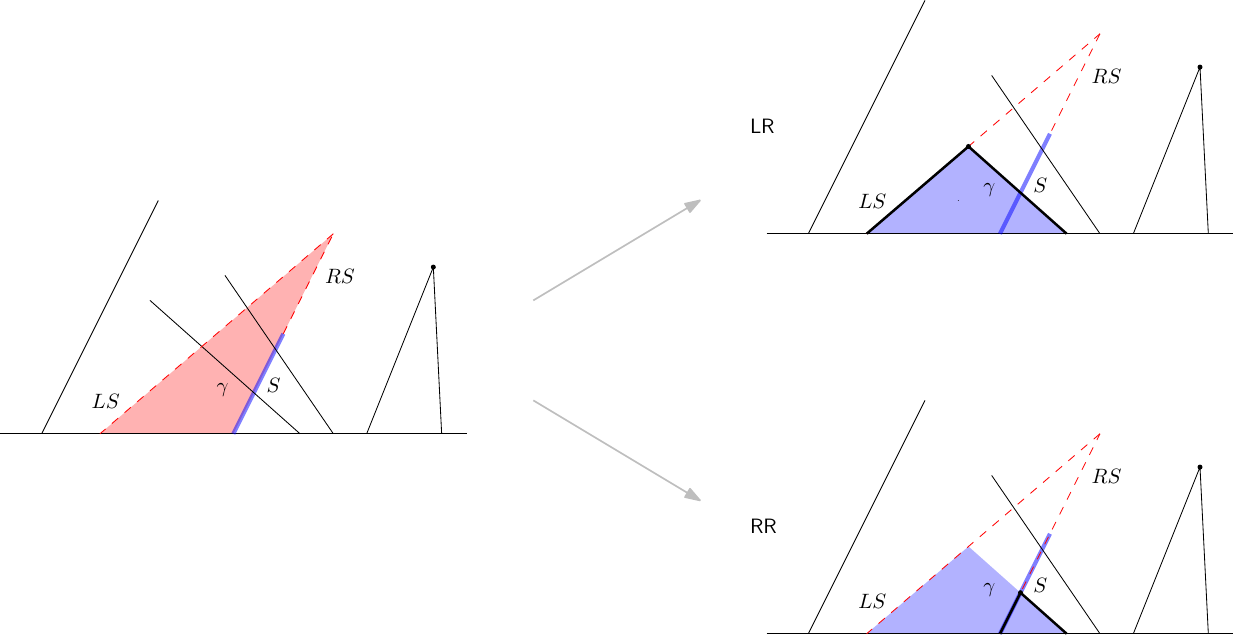}
\caption{\label{fig:firstpart}An illustration for the proof of Lemma~\ref{lem:ind2} (a): A configuration where the line $\gamma$ is not intersected before hitting $LS$: on the right we show the corresponding space-time diagrams in Models~\ref{def:LR} and~\ref{def:RR} where a bullet is shot along $LS$ or $RS$, respectively. It is important to note that except inside the blue shaded triangle, the two space-time diagrams are identical; as a consequence, the induction hypothesis is not necessary in this case.}
\end{figure}

  \item \emph{if $\gamma$ touches $RS$ and is intersected by some half line}, say $\gamma'$ before touching $LS$: in this case, we need the induction hypothesis to prove that, for all $k\ge 0$,
  \begin{equation}\label{eq:gamma'}
  \l|\LR(\bV_n,\bDelta_{n-1},s,\Z_+\setminus\{0\},k,\gamma,\gamma')\r|=\l|\RR(\bV_n,\bDelta_{n-1},s,\Z_+\setminus\{0\},k,\gamma,\gamma')\r|\,,
  \end{equation} 
where we added the entries $\gamma$ and $\gamma'$ to $\LR$ and to $\RR$ to denote the set of configurations corresponding to this situation. Any configuration $(\sigma,\tau)$ in $\cup_{k\ge 0}\LR(\bV_n,\bDelta_{n-1},s,\Z_+\setminus\{0\},k,\gamma,\gamma')$ is also a configuration of $\cup_{k\ge 0} \RR(\bV_n,\bDelta_{n-1},s,\Z_+\setminus\{0\},k,\gamma,\gamma')$, and reciprocally. Fix a configuration, that can be seen in both models. By assumption, we suppose that the half-lines $\gamma$ and $\gamma'$ intersect at some point $p$, before reaching $LS$  (see Figure~\ref{fig:secondpart}). Note that, in this configuration, any portion of the space-time diagram that is shot between $LS$ and $\gamma$, or between $\gamma$ and $\gamma'$ must remain trapped between $LS$ and $\gamma$, and $\gamma$ and $\gamma'$, respectively (see the two coloured regions in Figure~\ref{fig:secondpart}); this must be the case both in $\LR$ and $\RR$. As a consequence, none of the corresponding bullets may survive in any of the two models; for this reason, we can ignore them, and we now suppose that no bullet is shot between $LS$ and $\gamma$ or $\gamma$ and $\gamma'$.
  Now,
  \begin{itemize}
    \item in $\LR$, the two bullets whose trajectory follow $\gamma$ and $\gamma'$ indeed collide at the point $p$. The rest of the space-time diagram is a set of half-lines that does not intersect $\gamma'$ before the point $p$. The rest of the configuration is just constrained not to hit the portion of $\gamma'$ before $p$, that we call $S'$. There are $k$ surviving bullets precisely if this smaller configuration lies in $\LR(\bV'_{n-1}, \bDelta'_{n-2}, s', S', \{0\}, k)$, where $\bV_{n-1}'$, $\bDelta_{n-2}'$ are obtained by removing the necessary speeds, and merging the delays between $LS$ and $\gamma'$, and $s'$ denotes the second coordinate of $p$.

    \item in $\RR$, the bullet following $\gamma$ collides with the one following $RS$. The rest of the configuration has the same constraint that the line segment $S'$ of $\gamma'$ before the point $p$ is not intersected by any real trajectory. As a consequence, there are $k$ surviving bullets precisely if that smaller configuration lies in $\RR(\bV'_{n-1}, \bDelta_{n-2}', s', \{0\}, k)$, where it is important to note that $\bV'_{n-1}$, $\bDelta'_{n-2}$ and $s'$ are the same as above.
  \end{itemize}
  Next observe that there is a one-to-one map between the configurations of $n+1$ lines and the ones with $n-1$ lines\footnote{because we removed the ``trapped'' portions; otherwise it would be many-to-one, but the counting would still work since the constraints on the ``trap'' are the same in both models.}. It follows by the induction hypothesis that for each $k$, the sets $\LR(\bV'_{n-1}, \bDelta_{n-2}, s', \{0\}, k)$ and $\RR(\bV'_{n-1}, \bDelta_{n-2}, s', \{0\}, k)$ have the same cardinalities, and as a consequence, that \eqref{eq:gamma'} holds for all $k$.  
\end{enumerate}

\begin{figure}[htbp]
\centering
\includegraphics[width=16cm]{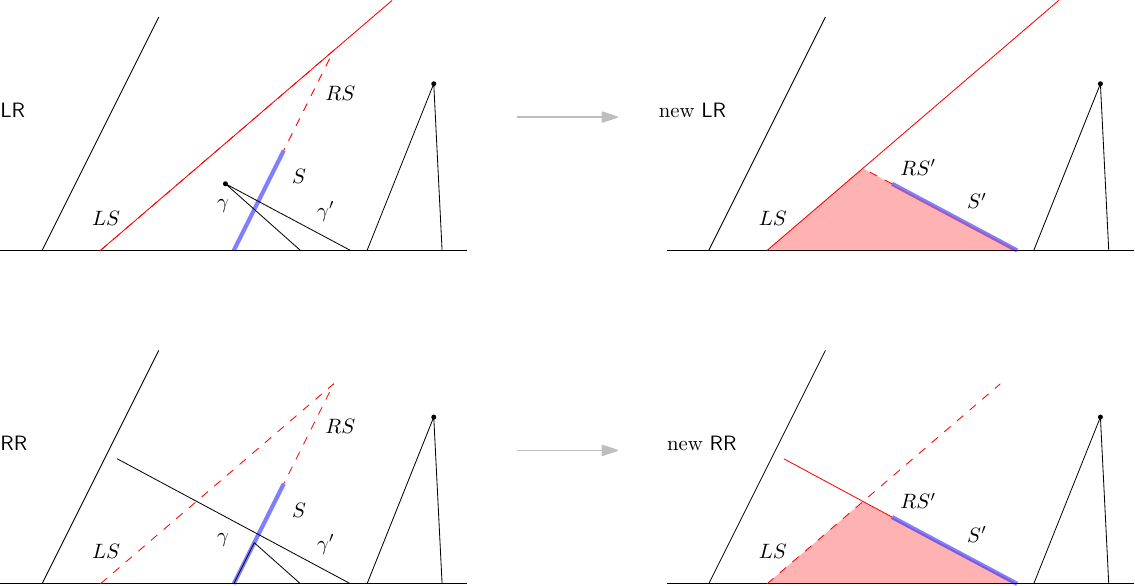}
\caption{\label{fig:secondpart}An illustration for the proof of Lemma~\ref{lem:ind2} (b): The space-time diagrams in Models~\ref{def:LR} and~\ref{def:RR} of a single configuration where $\gamma$ is hit by some line $\gamma'$ before hitting $LS$ (the one from Figure~\ref{fig:PLO}). One can transform the configuration into one for new matching $\LR$ and $\RR$ problems, where the critical pattern is modified. The new critical bi-triangle is shown shaded in red. The instance is of smaller size, and allows to use the induction hypothesis.}
\end{figure}

\subsection{Proof of Lemma~\ref{lem:ind3}}

As before, the general idea consists in proving that we can modify the parameters $(\bV_n, \bDelta_{n-1})$, making sure that we modify the statistics in either $\LR$ or $\RR$, until we arrive to a situation where we can without a doubt assert that these statistics are equal. In the previous proofs, the crucial modification consisted in decreasing the speed of the slowest bullet. Here, we rely on a different modification for the following reasons:
\begin{itemize}
	\item Although the models $\LR$ and $\RR$ seem to be images of one another under some natural symmetry, it is not the case. In particular, the fact that the slowest speed is the one that lies to the left of the special interval of length $\Delta^\star$ ruins the nice symmetries; for instance, no bullet may hit the slowest bullet from the left while it is certainly possible that a bullet hits the one shot right after $\Delta^\star$ from the right. 
	\item More importantly, recall that we said that the distribution of the number of surviving bullets \emph{does not in general remain the same} if the permutations $\sigma$ and $\tau$ of the speeds and delays are not independent. Here, the fact that the slowest speed is shot right before the interval of length $\Delta^\star$ creates a dependence which makes difficult to reduce the question to the initial colliding bullet problem where the intervals and speeds are all permuted independently.
\end{itemize}

Observe that 
 \[\LR(\bV_n,\bDelta_{n-1},0,\{0\}, k)
 \qquad \text{and} \qquad 
 \RR(\bV_n,\bDelta_{n-1},0,\{0\}, k)\] 
 are sets of configurations in a model where $n-1$ bullets are fired, but since the special segment $S$ has length zero (and no bullet with speed 0, for the case when $\min \bV_n=0$ can be treated directly), there are no restriction of any kind. In fact, we will prove a bit more than what is needed: we will prove that the equality of the statistics holds even if the speed $V_{\min}$ attached to the left of the special interval $\Delta^\star$ is any fixed speed (minimal or not). In other words, \emph{we do not assume that $V_{\min}$ is the minimal speed anymore in this section, but keep the name because the order of the speeds in $\bV_n$ has been defined with this name ($V_{\min}$ has been placed in second to last position). }

Since the permutations $(\sigma,\tau)$ we consider let attached $\Delta^\star$ and $V_{\min}$ (or $\Delta^\star$ and $V_r$), the statistics $\LR(\bV_n,\bDelta_{n-1},0,\{0\}, \cdot )$ (resp.\ $\RR(\bV_n, \bDelta_{n-1}, 0, \{0\}, \cdot)$) are only clearly given by $\bq_{n-1}$ when $\Delta^\star=0$, since this case corresponds exactly to $\bP^{\mC}_{\bV_{n-1},\bDelta_{n-2}}$ where $\bDelta_{n-2}$ is obtained from $\bDelta_{n-1}$ by removing $\Delta^\star$ and $\bV_{n-1}$ is obtained from $\bV_n$ by suppressing $V_r$ (resp.\ $V_{\min}$).

The idea of the proof is similar to that of Lemma~\ref{sub:proof_of_lemma2}, but rather than decreasing the minimal speed, we decrease the length $\Delta^\star$. Proceeding in this way addresses the two issues mentioned above: it does act symmetrically on $\LR$ and $\RR$ and, eventually, when $\Delta^\star=0$, the dependence between the permutations of the speeds and delays vanish. In the following, we consider only one of the problems $\LR$ or $\RR$: it important to understand that \emph{we do not compare directly $\LR$ and $\RR$} which seems difficult; we proceed by justifying that we can reduce $\Delta^\star$, without changing the statistics, until it reaches zero, at which point, $\LR$ and $\RR$ are identified. 
 
The core of the argument still relies on the kind of decompositions and reductions of the configurations that we have already treated in detail earlier and which should be familiar to the reader by now. So, since we have mentioned the main difficulty and the differences with the previous arguments, we allow ourselves to be quicker and only sketch the argument. 

When decreasing $\Delta^\star$ at the same time in $\LR$ and $\RR$, if the TCS is not modified, then the statistics are not modified. So we focus on the situations when the TCS does change: there exists configurations $(\sigma,\tau)$ for which one line from either side (before or after $\Delta^\star$) crosses the intersection of two lines originating from the other side. Only the configurations for which such a situation occurs need to be considered; furthermore, only the configuration for which the crossing indeed modifies the set of surviving bullets (at least locally) matter. This allows for the definition of notion of critical pattern or bi-triangle similar to the one is Section~\ref{ssec:singular-critical} (see Figure~\ref{fig:deuxieme_reduction}). As before, an induction argument allows to suppose that the critical pattern is minimal. Altogether, we are led to the situations described in Figure~\ref{fig:deuxieme_reduction}. Putting everything together, crossing a critical point of $\Delta^\star$ reduces to verifying that the statistics of $\LR$ and $\RR$ are identical for a problem of smaller size, and thus the induction hypothesis applies. 

\begin{figure}[htbp]
\centerline{\includegraphics[width=16cm]{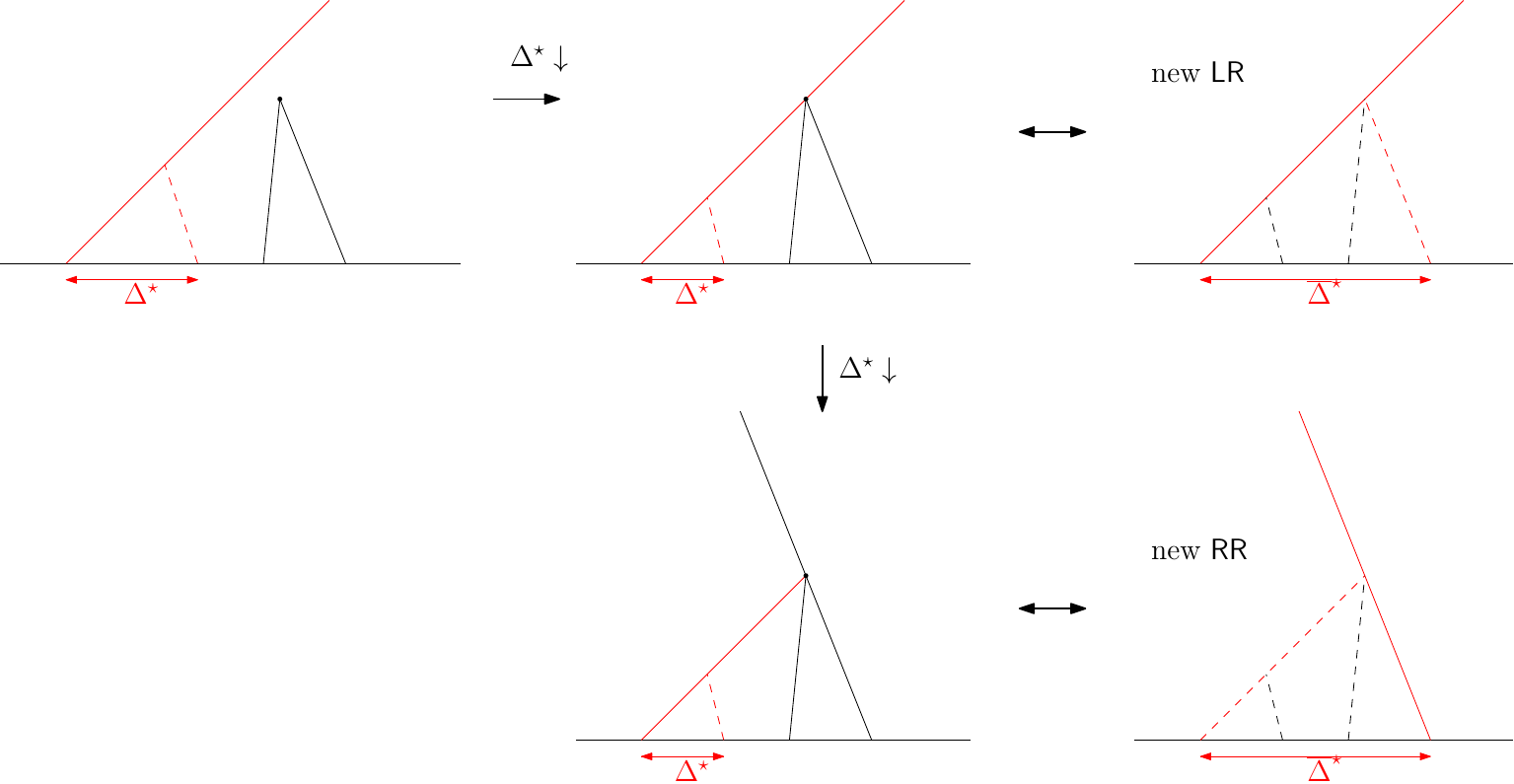}}
\caption{\label{fig:deuxieme_reduction}An illustration for the proof of Lemma~\ref{lem:ind3}: The typical effect of diminishing the value of $\Delta^\star$: one reaches a critical value for which there is a critical pattern involving three lines. The red lines are there to help comparing the statistics before and after the `crossing' of this critical value, which reduces to comparing the statistics of a couple of models $\LR$ / $\RR$ involving a smaller number of bullets.}
\end{figure}

\section{Remaining proofs}
\label{sec:pf_other}

\subsection{Proof of Theorem~\ref{theo:other}} 
\label{sub:proof_of_theorem_ref}

\noindent (i) \textsc{Bullet Flock.} Observe that the eventual number of surviving bullets only depends on the order of the speeds, not on their specific values. It does not depend either on the interbullet delays since the collisions are instantaneous. 
 Assume that the bullet with minimal speed is shot at time $i$ and that its speed is $v_i$. Being the slowest, it does not influence any of the bullets shot at times $1$, $2$, $\dots$, $i-1$. Furthermore, it will be hit by the next bullet $i+1$, except of course if there is no such bullet  that is, if $i=n$ which happens with probability $1/n$. Of course, the behaviour of the bullets shot from time $i+2$ is not influenced. Removing the bullets $i$, and $i+1$ when the latter exists leaves the system in an exchangeable situation again, with $n-1$ or $n-2$ bullets. The result follows readily from the recursive definition of $\bq_n$.

\medskip 
\noindent (ii) \textsc{Odd cycles in permutations}. We proceed again by induction on $n$.
Consider the element with label $1$ and the cycle ${\cal C}$ containing it. A straightforward computation shows that the length $|{\cal C}|$ of ${\cal C}$ is 1 with probability $1/n$ (in which case, $1$ is a fixed point). If $|{\cal C}|>1$, then let $b$ be the image of $1$ in the permutation (the number just after it around ${\cal C}$). By symmetry, $b$ is uniform in $\{2,\cdots,n\}$. Let also $c$ be the image of $b$, and $z$ the preimage of $1$. Note that we might have $c=1$ and $z=b$ if $|\mathcal C|=2$, or $z=c$ if $|\mathcal C|=3$. Consider now the cycle structure obtained as follows:
\begin{itemize}
  \item suppress the elements $1$ and $b$ from this cycle representation;
  \item if $z\ne b$, modify its image in the permutation so its new image is $c$.
\end{itemize}
A simple combinatorial argument shows that the remaining structure is the cycle representation of a permutation that is uniformly distributed in the symmetric group on $S'=\{1,2,\cdots,n\}\setminus \{1,b\}$: indeed, from a given permutation on $S'$ with cyclic representation $(C'_1,\cdots,C'_\ell)$ all the potential initial permutations can be obtained by 
\begin{itemize}
  \item either adding the cycle $(1,b)$ on its own,
  \item or inserting the linked pair $1,b$ to the right of any element in one of the cycles $C'_j$;
  the number of choices for the location of this insertion equals $\sum_j |C'_j|=n-2$, and thus does not depend on $(C_1',\dots, C_\ell')$.
\end{itemize}
This previous decomposition immediately yields the recurrence relation that defines $\bq_n$.


\medskip 
\noindent (iii) \textsc{Two-step directed tree}. In this case, the claim is straightforward since for every $n$, the distance to $0$ clearly satisfies the recurrence relation \eqref{eq:law_Xn} by construction.

\subsection{Proof of Proposition~\ref{pro:limit_dist}} 
\label{sub:limit_distribution}

The limit distribution for a random variable $X_n$ under $\bq_n$ follows from the combinatorial decomposition. Recall that, by Theorem~\ref{theo:other}, $X_n$ is distributed as the number of cycles of odd length in a uniformly random permutation of length $n$.
Since a permutation is a set of cycles, classical combinatorial decomposition \cite{FlSe2009a} yields that the bivariate generating function counting permutations where the cycles are marked by $u$ and size by $z$ is
\[P(z,u)= \sum_{n\ge 0} \sum_{k\ge 1} p_{n,k} u^k \frac{z^n}{n!} = \exp(- u \log (1-z)),\]
where $p_{n,k}$ is the number of size $n$ permutations with $k$ cycles multiplied by $n!$.
Here, we want the simple modification $P^\circ(z,u)$ of $P(z,u)$ where $u$ only marks the cycles of odd length, and we let $P(z,u)$ be the corresponding generating function. We find
\begin{align*}
P^\circ(z,u)
& =\exp\bigg( u \bigg[-\log(1-z) + \frac 1 2 \log(1-z^2)\bigg] - \frac 1 2 \log (1-z^2)\bigg)\\
& =(1+z)^{(u-1)/2} \cdot (1-z)^{-(1+u)/2}.
\end{align*}
The only two potential singularities are $z=\pm 1$, and as a function of $z$, the generating function $P^\circ(z,u)$ is clearly analytic in the following disk with two dents:
\[\mathcal D := \{z\in \mathbb C:  |z|\le 2, \arg(z-1)>\pi/12, \arg(1-z)>\pi/12 \}.\] 
If $u=1$, then $P^\circ(z,1)=1/(1-z)$ and there is a unique singularity at $1$; otherwise, for any $u$ in a complex punctured neighborhood $U$ of $1$, $P(z,u)$ has two singularities at $z\in\{+1,-1\}$. In this case, the domain $\cal D$ is what is referred in \cite{FlSe2009a} as a $\Delta$-domain and the singularity analysis transfer theorem implies that, as $n\to \infty$, the main contribution comes from $z=1$ (the other one has a lower order polynomial growth in $n$) and we have  
\[[z^n] P^\circ(z,u) = \frac1{n!} \sum_{k\ge 1} p^\circ_{n,k} u^k 
\sim \frac {2^{(u-1)/2}}{\Gamma(\frac{1+u}2)} n^{(u-1)/2}.
\]
We shall need a uniform estimate for $u\in U$, and we must look into the contribution of $(1+z)^{(u-1)/2}$ more carefully: standard binomial expansion yields the exact formula
\[[z^n](1+z)^{(u-1)/2} = \frac {\Gamma(\frac{1-u}2+n)}{\Gamma(\frac{1-u}2) \Gamma(n)}.\]
For any $\epsilon>0$, choosing $U$ to be the punctured ball of radius $2\epsilon$ around $1$, it follows that, uniformly in $U$, as $n\to\infty$,
\[\left|[z^n](1+z)^{(u-1)/2}\right| \le \frac{\Gamma(n+\epsilon)}{\Gamma(\epsilon) \Gamma(n+1)} \sim \frac{n^{\epsilon-1}}{\Gamma(\epsilon)}.\]
Standard manipulations then imply that the probability generating function $f_n(u)= \mathbf E[u^{X_n}]$ satisfies, again uniformly in $U$ provided that $\epsilon \in (0,1)$,
\[f_n(u) = \frac{[z^n] P^\circ(z,u)}{[z^n] P^\circ(z,1)} \sim \frac {2^{(u-1)/2}}{\Gamma(\frac{1+u}2)} e^{\frac 1 2(u-1) \log n}.\]
The quasi-powers theorem (Theorem IX.8 of \citet{FlSe2009a}) immediately yields that 
$\mathbf{E} X_n \sim \frac 1 2 \log n$, $\mathbf{Var}(X_n) \sim \frac 1 2 \log n$
and the claimed Gaussian limit distribution for $X_n$.

\black

\subsection{Proof of Proposition~\ref{pro:recurrence}} %
\label{sub:infinite_cases}

{\em (i)} Let $T_x$ denote the time (number of bullets to shoot) to destroy the slowest bullet in the flock given that it has a given speed $x\in [0,1]$. Considering the speed of the bullet that is first shot yields the following integral equation for $T_x$:
\begin{equation}
\label{eq:flock_diff_eq}
\E({T_x}) = 1 + x \left(\E(T_{Ux}) + \E({T_x})\right),
\end{equation}
where $U$ is an random variable uniform on $[0,1]$. This is a simple differential equation, and one obtains, with the condition that $T_0=1$ almost surely,
\[\E({T_x})=  1 /{(1-x)^2}.\]
This shows in particular that for every $x$, $T_x$ is almost surely finite, and in turn, that the time to destroy any finite number of bullets is also finite. This ensures that $F_n=0$ infinitely often.

{\em (ii)} For the case of the distances in the two-step tree, one easily verifies that the two subsequences $(D_{2k})_{k\ge 1}$ and $(D_{2k+1})_{k\ge 0}$ are non-decreasing. The convergence in probability thus implies almost sure convergence, and in turn the fact that there exists a (random) integer $n_0$ such that $D_n\ge 1$ for all $n\ge n_0$, which proves the claim.

\bibliographystyle{plainnat}  

\small
\setlength{\bibsep}{0.25em}
\bibliography{bibi}

\end{document}